\documentclass{amsart}
\usepackage{dsfont}
\usepackage{mathrsfs}
\usepackage{color}
\usepackage{CJK}
\usepackage[utf8]{inputenc}
\usepackage{mathtools}
\usepackage{color}
\usepackage[english]{babel}

\newtheorem{theorem}{Theorem}[section]
\newtheorem{lemma}[theorem]{Lemma}

\theoremstyle{definition}
\newtheorem{definition}{Definition}[section]
\newtheorem{example}[theorem]{Example}

\theoremstyle{remark}
\newtheorem{remark}{Remark}[section]

\numberwithin{equation}{section}

\makeatletter
\newcommand{\rmnum}[1]{\romannumeral #1}
\newcommand{\Rmnum}[1]{\expandafter\@slowromancap\romannumeral #1@}
\makeatother

\begin{document}

\title{On the measure-theoretic entropy and topological pressure of free semigroup actions(to appear in ETDS)}

\author{Xiaogang Lin}
\address{School of Business Administration, South China University of Technology,
Guangzhou 510641, P.R. China}
\email{XGL1010@foxmail.com}

\author{Dongkui Ma*}
\thanks{* Corresponding author}
\address{School of Mathematics, South China University of Technology,
Guangzhou 510641, P.R. China}
\email{dkma@scut.edu.cn}

\author{Yupan Wang}
\address{School of Computer Science and Engineering, South China University of Technology,
Guangzhou 510641, P.R. China}





\begin{abstract}
In this paper, we introduce the notions of topological pressure and measure-theoretic entropy for a free semigroup action. Suppose that a free semigroup acts on a compact metric space by continuous self-maps. To this action,  we assign a skew-product transformation whose fiber topological pressure is taken to be the topological pressure of the initial action. Some properties of these two notions are given, and then we give two main results. One is the relationship between the topological pressure of the skew-product transformation and the topological pressure of the free semigroup action, the other is the partial variational principle about the topological pressure. Moreover, we apply this partial variational principle to study the measure-theoretic entropy and the topological entropy of finite affine transformations on a metrizable group.
\end{abstract}

\maketitle

\section{ Introduction }

In 1959, Kolmogorov and Sinai developed the metric entropy from Shannon's information theory into ergodic theory. Since then, the notion of entropy has played a vital role in the study of dynamic systems. In 1967, Kirillov\cite{KIR} introduced the notion of entropy for the action of finitely generated groups of measure-preserving transformations.  The case of Abelian group actions was studied by Conze\cite{CONZE} and also by Katznelson and Weiss\cite{KAT}. The notion of topological entropy, which was introduced by Adler, Konheim and McAndrew\cite{AKM} as an invariant of topological conjugacy, describes the complexity of a system. Later, Bowen\cite{Bowen} and Dinaburg\cite{Dinaburg} provided an equivalent definition when the space is  metrizable. Since the entropy appeared to be a very useful invariant in ergodic theory and dynamical systems, there were several attempts to find its suitable generalizations for other systems such as groups, pseudogroups, graphs, foliations, nonautonomous dynamical systems and so on(\cite{BIS1,BIS3,BIS4,BIS2,BUfE,FRI,GLW,KOL,LSHMA,MAWU,WMA}). Bi\'{s}\cite{BIS1} and Bufetov\cite{BUfE} introduced the notion of the topological entropy of free semigroup actions. Related studies include\cite{BIS2, CARVALHO,LSHMA,MAWU,RODRIGUES,TANG,WMA,WMAL}.

As a natural extension of topological entropy, topological pressure is a rich source of dynamic systems. Ruelle\cite{RUE} first introduced the concept of topological pressure of additive potentials for expansive dynamic systems. Walters\cite{WALTER1} then extended this concept to a compact space with the continuous transformation. Since  topological pressure appeared to be useful in ergodic theory and dynamic systems, there were several attempts to find its suitable generalizations for other systems(see, for example, \cite{ BAR,BIS3,CAO,CHUNG,HUANG,LSHMA,MAWU,PESIN,THOM,ZHANG}).

In general, it's difficult to study the properties of dynamical systems for arbitrary group actions, so one needs to find some special groups which meet certain conditions, e.g., amenable group, sofic group, etc. In this paper, we assume that a semigroup is free because the topological pressure(entropy) for a free semigroup action can associate with the topological pressure(entropy) for the skew-product dynamical system, and thus we can study the skew-product dynamical system based on the properties of the dynamical system for a free semigroup action.

In this paper, we introduce the notions of the topological pressure and the measure-theoretic entropy for a free semigroup action and give some properties of these two notions.

Let $X$ be a compact metric space, $f_0,\cdots,f_{m-1}$ continuous self-maps acting on $X$ and $F:\Sigma_m \times X\rightarrow \Sigma_m \times X$  a skew-product transformation, where $\Sigma_m$ denotes the two-side symbol space. Let $\varphi \in C(X,\mathbb{R})$ and $g =c+\varphi$, where $c$ is a constant. $P(F,g)$ denotes the topological pressure of $F$ with respect to $g$ for the dynamical system $(\Sigma_m \times X,F)$, $P(f_0,\cdots,f_{m-1},\varphi)$ the topological pressure of $f_0,\cdots,f_{m-1}$ with respect to $\varphi$ and $h(f_0,\cdots,f_{m-1})$  the topological entropy of $f_0,\cdots,f_{m-1}$, and $h_{\mu}(f_0,\ldots,f_{m-1})$  the measure-theoretic entropy of  $f_0,\cdots,f_{m-1}$ with respect to the measure $\mu$ for the free semigroup action. $ \mathscr{M}(X,f_0,\ldots,f_{m-1})$ denotes the set of invariant measures of $f_0,\cdots,f_{m-1}$.

The main results of this paper are the following two theorems.
\begin{theorem}\label{thm:properties of pressure3}
The topological pressure of the transformation $F$ with respect to $g$, satisfies \[P(F,g)=c+\log m+P(f_0,\cdots,f_{m-1},\varphi).\]
\end{theorem}

\begin{remark}
If $c=0$, then $P(F,g)=\log m+P(f_0,\cdots,f_{m-1},\varphi)$. If $g=\varphi\equiv 0$, then \[h(F)=\log m+h(f_0,\cdots,f_{m-1})\] which has been proved by Bufetov\cite{BUfE}.
\end{remark}

\begin{theorem}\label{thm:relation1}
(Partial variational principle)Let $(X,d)$ be a compact metric space and let $f_i:X\rightarrow X$(i=0,\ldots,m-1) be  finite continuous maps such that $\mathscr{M}(X,f_0,\ldots,f_{m-1})\neq \emptyset$ and let $\varphi \in C(X,\mathbb{R})$. Then

 \[\sup \left\{h_{\mu}(f_0,\ldots,f_{m-1})+\int \varphi d\mu \Big| \mu \in \mathscr{M}(X,f_0,\ldots,f_{m-1})\right\} \leq P(f_0,\ldots,f_{m-1},\varphi).\] In particular, when $\varphi=0$, we have \[\sup\left\{h_{\mu}(f_0,\ldots,f_{m-1}) | \mu \in \mathscr{M}(X,f_0,\ldots,f_{m-1})\right\} \leq h(f_0,\ldots,f_{m-1}).\]
\end{theorem}

This paper is organized as follows. In section 2, we give some preliminaries. In section 3, we introduce the notion of the topological pressure for a free semigroup action on a compact metric space and give some fundamental properties. In section 4, we give the definition and some properties of the measure-theoretic entropy for a free semigroup action on a probability space. In section 5, we prove our main results. In section 6, we apply Theorem \ref{thm:relation1} to study the measure-theoretic entropy and the topological entropy for finite affine transformations.

\section{preliminaries}

\subsection{Words and sequences}

Before studying the topological pressure and the measure-theoretic entropy for a free semigroup action, we introduce some notations. Denote by $F_m^{+}$ the set of all finite words of symbols $0,1,\ldots,m-1$. For any $w \in F_m^{+}$, $ |w|$ stands for the length of $w$, that is, the number of symbols in $w$. Obviously, $F_m^{+}$ with respect to this law of composition is a free semigroup with $m$ generators. We write $w\leq w'$ if there exists a word $w'' \in F_m^{+}$ such that $w'=w''w$.

Denote by $\Sigma_m$ the set of all two-side infinite sequences of symbols $0,1,\ldots,m-1$, i.e., \[ \Sigma_m=\{\omega=(\ldots, \omega_{-1}, \omega_0, \omega_1, \ldots) | \omega_i=0,1,\ldots,m-1~for~all~integer~i\}.\]

A metric on $\Sigma_m$ is introduced by  \[d(\omega,\omega')=1/2^k,~where~k=\inf\{|n|:\omega_n\neq\omega'_n\}.\]

Obviously, $\Sigma_m$ is compact with respect to this metric. Recall that the Bernoulli shift $\sigma_m: \Sigma_m \rightarrow \Sigma_m$ is a homeomorphism of $\Sigma_m$ given by the formula: \[(\sigma_m\omega)_i=\omega_{i+1}.\]

Let $\omega \in \Sigma_m,w\in F_m^{+}$, $a,b$ integers, and $a\leq b$. We write $\omega|_{[a,b]}=w$ if $w=\omega_a\omega_{a+1}\ldots\omega_{b-1}\omega_b$.

\subsection{Topological entropy for a free semigroup action}\label{def:topological}

In this section, we recall the topological entropy for a free semigroup action on a compact metric space. Our presentation follows Bufetov\cite{BUfE}.

Let $X$ be a compact metric space with metric $d$. Suppose that a free semigroup with $m$ generators acts on $X$; denote the maps corresponding to the generators by $f_0,f_1,\ldots,f_{m-1}$; we assume that these maps are continuous.

Let $w \in F_m^{+}, w=w_1w_2\ldots w_k$, where $w_i=0,1,\ldots,m-1$ for all $i=1,\ldots,k$. Let $f_w=f_{w_1}f_{w_2}\ldots f_{w_k}$,  $f_w^{-1}=f_{w_k}^{-1}f_{w_{k-1}}^{-1}\ldots f_{w_1}^{-1}$.  Obviously, $f_{ww'}=f_wf_{w'}$.

To each $w \in F_m^{+}$, a new metric $d_w$ on $X$ (named Bowen metric) is given by \[d_w(x_1,x_2)=\max_{w' \leq w} d(f_{w'}(x_1),f_{w'}(x_2)).\]

Clearly, if $w \leq w'$, then $d_w(x_1,x_2) \leq d_{w'}(x_1,x_2)$ for all $x_1,x_2 \in X$.

Let $\varepsilon>0$, a subset $E$ of $X$ is said to be a $(w,\varepsilon,f_0,\ldots,f_{m-1})$-spanning subset if, for $\forall x\in X, \exists y \in E$ with $d_w(x,y)<\varepsilon$. The minimal cardinality of a $(w,\varepsilon,f_0,\ldots,f_{m-1})$-spanning subset of $X$ is denoted by $B(w,\varepsilon,f_0,\ldots,f_{m-1})$.

Let $\varepsilon>0$, a subset $F$ of $X$ is said to be $(w,\varepsilon,f_0,\ldots,f_{m-1})$-separated subset if, for any $ x_1,x_2 \in F, x_1\neq x_2$, $d_w(x_1,x_2) \geq \varepsilon$. The maximal cardinality of any $(w,\varepsilon,f_0,\ldots,f_{m-1})$-separated subsets of $X$ is denoted by $N(w,\varepsilon,f_0,\ldots,f_{m-1})$.

Let \[B(n,\varepsilon,f_0,\ldots,f_{m-1})= \frac{1}{m^n}\sum_{|w|=n}B(w,\varepsilon,f_0,\ldots,f_{m-1}),\]
\[N(n,\varepsilon,f_0,\ldots,f_{m-1})= \frac{1}{m^n}\sum_{|w|=n}N(w,\varepsilon,f_0,\ldots,f_{m-1}).\]

The topological entropy of a free semigroup action is defined by the formula \[h(f_0,\ldots,f_{m-1})=\lim_{\varepsilon \rightarrow 0}(\limsup_{n \rightarrow \infty})\frac{1}{n} \log B(n,\varepsilon,f_0,\ldots,f_{m-1}). \]

It follows that \[h(f_0,\cdots,f_{m-1})=\lim_{\varepsilon \rightarrow 0}(\limsup_{n \rightarrow \infty})\frac{1}{n} \log N(n,\varepsilon,f_0,\cdots,f_{m-1}). \]

We introduce the definition of topological entropy in the sense of Bi\'{s}\cite{BIS1} and give the relationship between it and the topological entropy in the sense of Bufetov. Let $G_1=\{id_X, f_0,\cdots,f_{m-1}\}, G=\cup_{n \in \mathbb{N}}G_n$, where $G_n=\{g_1\circ\cdots\circ g_n:g_1,\cdots,g_n\in G_1\}$. We say that two point $x,y\in X$ are $(n,\varepsilon)$-separated by $G$ if there exists $g\in G_n$ such that $d(g(x),g(y))\geq \varepsilon$. A set $A$ of X is called a $(n,\varepsilon)$-separated set in the sense of Bi\'{s} if any two distinct points of $A$ have this property. Let
\[s(n,\varepsilon,X)=\max\{card(E):E~~is~~a~~(n,\varepsilon)-seperated~~subset~~of~~X\}.\]
In \cite{BIS1}, Bi\'{s} defined the topological entropy as
\[h(G,G_1,X)=\lim_{\varepsilon\rightarrow0}\limsup_{n\rightarrow\infty}\frac{1}{n}\log(s(n,\varepsilon,X)).\]

\begin{remark}
(1) It is easy to see that $h(f_0,\cdots,f_{m-1})\leq h(G,G_1,X)$.

(2) Observe that if $m=1$, this definition coincides with Bowen's definition for the dynamical system $(X,f)$\cite{Bowen,WALTER2}.
\end{remark}

\section{Topological pressure for a free semigroup action}\label{sec:pressureDefinition}
In this section, the topological pressure  for a free semigroup action is introduced. Some properties are also given.

Let $X$ be a compact metric space with metric $d$. Suppose that a free semigroup with $m$ generators act on $X$; denote the maps corresponding to the generators by $f_0,\cdots,f_{m-1}$; we assume that these maps are continuous.
Let $C(X,\mathbb{R})$ be the space of all the real-valued continuous functions of $X$. For $w \in F_m^{+}$, $w'\leq w$ and $\varphi \in C(X,\mathbb{R})$, we denote $\sum_{w'\leq w}\varphi(f_{w'}x)$ by $(S_{w}\varphi)(x)$.

\subsection{Definition using spanning sets}\label{subsec:spanning sets}
\begin{definition}\label{def:spanning sets1}
For $w \in F_m^{+}$, $\varphi \in C(X,\mathbb{R})$, $n \geq 1$ and $\varepsilon>0$ put
\begin{align*}
&\quad Q_w(f_0,\ldots,f_{m-1},\varphi,\varepsilon)\\&=\inf\{\sum_{x \in F}e^{(S_{w}\varphi)(x)}| F~is~a~ (w,\varepsilon,f_0,\ldots,f_{m-1})~spanning~set~of~X \},
\end{align*}
\[Q_n(f_0,\ldots,f_{m-1},\varphi,\varepsilon)=\frac{1}{m^n}\sum_{|w|=n} Q_w(f_0,\ldots,f_{m-1},\varphi,\varepsilon).\]
\end{definition}

\begin{remark} \label{remark:spanning sets1}
According to the foregoing description, the following statements are true.
$\\$

(1) $0<Q_w(f_0,\ldots,f_{m-1},\varphi,\varepsilon) \leq \|e^{S_{w}\varphi}\|B(w,\varepsilon,f_0,\ldots,f_{m-1})< \infty$, where  $\|\varphi\|=\max_{x \in X}|\varphi|$. Hence, \[0<Q_n(f_0,\ldots,f_{m-1},\varphi,\varepsilon) \leq e^{n\|\varphi\|}B(n,\varepsilon,f_0,\ldots,f_{m-1})< \infty.\]

(2) If $\varepsilon_1 < \varepsilon_2$, then $Q_w(f_0,\ldots,f_{m-1},\varphi,\varepsilon_1)\geq Q_w(f_0,\ldots,f_{m-1},\varphi,\varepsilon_2)$. Hence, \[Q_n(f_0,\ldots,f_{m-1},\varphi,\varepsilon_1)\geq Q_n(f_0,\ldots,f_{m-1},\varphi,\varepsilon_2).\]

(3) $Q_w(f_0,\ldots,f_{m-1},0,\varepsilon)=B(w,\varepsilon,f_0,\ldots,f_{m-1})$. Hence, \[Q_n(f_0,\ldots,f_{m-1},0,\varepsilon)=B(n,\varepsilon,f_0,\ldots,f_{m-1}).\]

(4)In Definition \ref{def:spanning sets1} it suffices to take the infinium over those spanning sets which don't have proper subsets that $(w,\varepsilon)$ span $X$. This is because $e^{(S_{w}\varphi)(x)}>0$.
\end{remark}

Set \[Q(f_0,\ldots,f_{m-1},\varphi,\varepsilon)=\limsup_{n\rightarrow \infty} \frac{1}{n}\log Q_n(f_0,\ldots,f_{m-1},\varphi,\varepsilon).\]

\begin{remark}\label{remark:spanning sets2}
$\\$

(1) $Q(f_0,\ldots,f_{m-1},\varphi,\varepsilon) \leq \|\varphi\|+\limsup_{n\rightarrow \infty} \frac{1}{n}\log B(n,\varepsilon,f_0,\ldots,f_{m-1})< \infty$.

(2) If $\varepsilon_1 < \varepsilon_2$, then $Q(f_0,\ldots,f_{m-1},\varphi,\varepsilon_1)\geq Q(f_0,\ldots,f_{m-1},\varphi,\varepsilon_2)$ (according to Remark \ref{remark:spanning sets1}(2)), i.e., $Q(f_0,\ldots,f_{m-1},\varphi,\varepsilon)$ is monotonous with respect to $\varepsilon$.

\end{remark}

\begin{definition}\label{def:spanning sets2}
For $\varphi \in C(X,\mathbb{R})$, the topological pressure for a free semigroup action with respect to $\varphi$ is defined as \[P(f_0,\ldots,f_{m-1},\varphi)=\lim_{\varepsilon\rightarrow 0}Q(f_0,\ldots,f_{m-1},\varphi,\varepsilon).\] If we let $G:=\{f_0,\cdots,f_{m-1}\}$, then we also denote $P(f_0,\ldots,f_{m-1},\varphi)$ by $P(G,\varphi)$.
\end{definition}

When $m=1$, let $f_0=f$, this definition is just the topological pressure $P(f,\varphi)$ for dynamical system $(X,f)$\cite{WALTER2}. When $\varphi=0$, $P(f_0,\ldots,f_{m-1},\varphi)$ is just the topological entropy for a free semigroup action defined in Section \ref{def:topological} and is denoted by $h(f_0,\ldots,f_{m-1})$\cite{BUfE}. Furthermore, if $m=1$ and $\varphi=0$, then it is trivial that $P(f_0,\ldots,f_{m-1},\varphi)$ is the topological entropy of $f$(usually denoted by $h(f)$) for dynamical system $(X,f)$\cite{Bowen,WALTER2}.

\begin{remark}\label{remark:spanning sets3}
By Remark \ref{remark:spanning sets2}(2), $P(f_0,\ldots,f_{m-1},\varphi)$ exists, but could be $\infty$. For example, when $m=1,\varphi=0$, we have $P(f_0,\ldots,f_{m-1},\varphi)=h(f_0)$, i.e., the classical topological entropy in the sense of Bowen. Many literatures show that $h(f_0)=\infty$, e.g.,\cite{BUGR}.
\end{remark}

\subsection{Definition using separated sets}\label{subsec:separated sets}

\begin{definition}\label{def:separated sets1}
For $w \in F_m^{+}$, $\varphi \in C(X,\mathbb{R})$, $n \geq 1$ and $\varepsilon>0$ put
\begin{align*}
&\quad P_w(f_0,\ldots,f_{m-1},\varphi,\varepsilon)\\&=\sup\{\sum_{x \in E}e^{(S_{w}\varphi)(x)}| E~is~a~ (w,\varepsilon,f_0,\ldots,f_{m-1})~separated~subset~of~X \},
\end{align*}
\[P_n(f_0,\ldots,f_{m-1},\varphi,\varepsilon)=\frac{1}{m^n} \sum_{|w|=n} P_w(f_0,\ldots,f_{m-1},\varphi,\varepsilon).\]
\end{definition}

\begin{remark}\label{remark:separated sets1}
$\\$

(1) If $\varepsilon_1 < \varepsilon_2$, then $P_w(f_0,\ldots,f_{m-1},\varphi,\varepsilon_1)\geq P_w(f_0,\ldots,f_{m-1},\varphi,\varepsilon_2)$. Hence, \[P_n(f_0,\ldots,f_{m-1},\varphi,\varepsilon_1)\geq P_n(f_0,\ldots,f_{m-1},\varphi,\varepsilon_2).\]

(2) $P_w(f_0,\ldots,f_{m-1},0,\varepsilon)=N(w,\varepsilon,f_0,\ldots,f_{m-1})$. Hence, \[P_n(f_0,\ldots,f_{m-1},0,\varepsilon)=N(n,\varepsilon,f_0,\ldots,f_{m-1}).\]

(3)In Definition \ref{def:separated sets1} it suffices to take the supremum over all $(w,\varepsilon)$-separated sets which fail to be $(w,\varepsilon)$-separated when any point of $X$ is added. This is because $e^{(S_{w}\varphi)(x)}>0$.

(4) $Q_w(f_0,\ldots,f_{m-1},\varphi,\varepsilon)\leq P_w(f_0,\ldots,f_{m-1},\varphi,\varepsilon)$. This follows from Remark (3) and the fact that a $(w,\varepsilon,f_0,\ldots,f_{m-1})$-separated set which can not be enlarged to a $(w,\varepsilon,f_0,\ldots,f_{m-1})$-separated set must be a $(w,\varepsilon,f_0,\ldots,f_{m-1})$-spanning set of $X$. Hence \[Q_n(f_0,\ldots,f_{m-1},\varphi,\varepsilon)\leq P_n(f_0,\ldots,f_{m-1},\varphi,\varepsilon).\]

(5) If $\delta=\sup\{|\varphi(x)-\varphi(y)| : d(x,y)\leq \frac{\varepsilon}{2}\}$, then \[P_n(f_0,\ldots,f_{m-1},\varphi,\varepsilon)\leq e^{n\delta}Q_n(f_0,\ldots,f_{m-1},\varphi,\frac{\varepsilon}{2}) .\]
\begin{proof}
For $w \in F_m^{+}, |w|=n$, let $E$ be a $(w,\varepsilon,f_0,\ldots,f_{m-1})$ separated set and $F$ a $(w,\frac{\varepsilon}{2},f_0,\ldots,f_{m-1})$ spanning set. Define $\phi : E \rightarrow F$ by choosing, for each $x \in E$, some point $\phi(x) \in F$ with $d_w(x,\phi(x))\leq \frac{\varepsilon}{2}$(using the notation $d_w(x,y)=\max_{w'\leq w}d(f_{w'}(x),f_{w'}(y))$). Then $\phi$ is injective so
\begin{align*}
\sum_{y \in F}e^{(S_{w}\varphi)(y)} \geq \sum_{y \in \phi E}e^{(S_{w}\varphi)(y)} &\geq (\min_{x \in E} e^{(S_{w}\varphi)(\phi x)-(S_{w}\varphi)(x)})\sum_{x \in E}e^{(S_{w}\varphi)(x)}\\
& \geq e^{-|w|\delta}\sum_{x \in E}e^{(S_{w}\varphi)(x)}.
\end{align*}
Therefore $P_w(f_0,\ldots,f_{m-1},\varphi,\varepsilon)\leq e^{|w|\delta}Q_w(f_0,\ldots,f_{m-1},\varphi,\frac{\varepsilon}{2})$. Obviously, we have \[P_n(f_0,\ldots,f_{m-1},\varphi,\varepsilon)\leq e^{n\delta}Q_n(f_0,\ldots,f_{m-1},\varphi,\frac{\varepsilon}{2}).\]
\end{proof}
\end{remark}

Then, set \[P(f_0,\ldots,f_{m-1},\varphi,\varepsilon)=\limsup_{n\rightarrow \infty} \frac{1}{n}\log P_n(f_0,\ldots,f_{m-1},\varphi,\varepsilon).\]

\begin{remark}\label{remark:separated sets2}
As above, the following statements are true.
$\\$

(1) $Q(f_0,\ldots,f_{m-1},\varphi,\varepsilon)\leq P(f_0,\ldots,f_{m-1},\varphi,\varepsilon)$(according to Remark \ref{remark:separated sets1}(4)).

(2) If $\delta=\sup\{|\varphi(x)-\varphi(y)| : d(x,y)\leq \frac{\varepsilon}{2}\}$, then \[P(f_0,\ldots,f_{m-1},\varphi,\varepsilon)\leq \delta+Q(f_0,\ldots,f_{m-1},\varphi,\varepsilon) \](according to Remark \ref{remark:separated sets1}(5)).

(3) If $\varepsilon_1 < \varepsilon_2$, then  \[P(f_0,\ldots,f_{m-1},\varphi,\varepsilon_1)\geq P(f_0,\ldots,f_{m-1},\varphi,\varepsilon_2)\](according to Remark \ref{remark:separated sets1}(1)).
\end{remark}

\begin{theorem}\label{thm:spanning is separated}
If $\varphi \in C(X,\mathbb{R})$, then \[P(f_0,\ldots,f_{m-1},\varphi)=\lim_{\varepsilon\rightarrow 0}P(f_0,\ldots,f_{m-1},\varphi,\varepsilon).\]
\end{theorem}

\begin{proof}
The limit exists by Remark \ref{remark:separated sets2}(3). By Remark \ref{remark:separated sets2}(1) we know that \[P(f_0,\ldots,f_{m-1},\varphi)\leq \lim_{\varepsilon\rightarrow 0}P(f_0,\ldots,f_{m-1},\varphi,\varepsilon).\] By Remark \ref{remark:separated sets2}(2), for any $\delta >0$, we have \[\lim_{\varepsilon\rightarrow 0}P(f_0,\ldots,f_{m-1},\varphi,\varepsilon)\leq \delta +P(f_0,\ldots,f_{m-1},\varphi),\] which implies \[\lim_{\varepsilon\rightarrow 0}P(f_0,\ldots,f_{m-1},\varphi,\varepsilon)\leq P(f_0,\ldots,f_{m-1},\varphi).\]
\end{proof}

We introduce the definition of topological pressure in the sense of Ma and Liu\cite{LSHMA} and give the relationship between it and the topological pressure we give. Let $G_1=\{id_X, f_0,\cdots,f_{m-1}\}$ and
\begin{align*}
&\quad A(\varphi,G_1,\varepsilon,n)\\&:=\sup\{\Sigma_{x \in E}e^{(S_n\varphi)(x)}:E~~is~~a~~(n,\varepsilon)-separated~~subset~~in~~the~~sense~~of~~\text{Bi\'s} \}.
\end{align*}
In \cite{LSHMA}, the topological pressure
\[P_X(\varphi,G_1)=\lim_{\varepsilon\rightarrow0}\limsup_{n\rightarrow\infty}\frac{1}{n}\log A(\varphi,G_1,\varepsilon,n).\]

\begin{remark}
(1) It is easy to see that $P(f_0,\ldots,f_{m-1},\varphi)\leq P_X(\varphi,G_1)$.

(2) Recently, in the course of revising our paper, we find that Rodrigues and Varandas\cite{RODRIGUES} also introduce the definition of topological pressure using separated sets.
\end{remark}

\subsection{Definition using open covers}\label{subsec:open cover}

Let $(X,d)$ be a compact metric space and $\alpha$ be an open cover of $X$. Let $w \in F_m^{+}, w=w_1w_2\ldots w_k$, where $w_i=0,1,\ldots,m-1$ for all $i=1,\ldots,k$. Recall that $w' \leq w$ if there exists a word $w'' \in F_m^{+}$ such that $w=w''w'$. for any $w' \leq w\in F_m^{+}$, denote $f_{w'}^{-1}(\alpha)=\{f_{w'}^{-1}(A) | A \in \alpha \}$. Let $\alpha,\beta$ be two open covers of $X$, we use $\alpha\vee\beta$ denote the open cover by all sets of the form $A\cap B$ where $A \in \alpha,B \in \beta$. We use $\alpha<\beta$ denote $\beta$ is a refinement of $\alpha$, i.e., every member of $\beta$ is a subset of a member of $\alpha$. If $\varphi \in C(X,\mathbb{R})$, then set

\begin{align*}
&\quad q_w(f_0,\ldots,f_{m-1},\varphi,\alpha)\\&=\inf \left\{\sum_{B \in \beta}\inf_{x \in B} e^{(S_{w}\varphi)(x)}| \beta~is~a~finite~subcover~of~\bigvee_{w'\leq w}f_{w'}^{-1}(\alpha) \right\},
\end{align*}
\[q_n(f_0,\ldots,f_{m-1},\varphi,\alpha)=\frac{1}{m^n}\sum_{|w|=n} q_w(f_0,\ldots,f_{m-1},\varphi,\alpha).\]
and
\begin{align*}
&\quad p_w(f_0,\ldots,f_{m-1},\varphi,\alpha)\\&=\inf\left\{\sum_{B \in \beta}\sup_{x \in B} e^{(S_{w}\varphi)(x)}| \beta~is~a~finite~subcover~of~\bigvee_{w'\leq w}f_{w'}^{-1}(\alpha) \right\},
\end{align*}
\[p_n(f_0,\ldots,f_{m-1},\varphi,\alpha)=\frac{1}{m^n}\sum_{|w|=n} p_w(f_0,\ldots,f_{m-1},\varphi,\alpha).\]
Clearly, $q_w(f_0,\ldots,f_{m-1},\varphi,\alpha)\leq p_w(f_0,\ldots,f_{m-1},\varphi,\alpha)$, and $q_n(f_0,\ldots,f_{m-1},\varphi,\alpha)\leq p_n(f_0,\ldots,f_{m-1},\varphi,\alpha)$.

\begin{theorem}\label{thm:open cover1}
Let $f_i:X\rightarrow X$ be continuous, $i=0,\ldots,m-1$, and $\varphi \in C(X,\mathbb{R})$.

$(\rmnum{1})$ If $\alpha$ is an open cover of $X$ with Lebesgue number $\delta$, then $q_n(f_0,\ldots,f_{m-1},\varphi,\alpha)\leq Q_n(f_0,\ldots,f_{m-1},\varphi,\frac{\delta}{2}) \leq P_n(f_0,\ldots,f_{m-1},\varphi,\frac{\delta}{2})$.

$(\rmnum{2})$ If $\varepsilon>0$ and $\gamma$ is an open cover with $diam(\gamma)\leq \varepsilon$, then $Q_n(f_0,\ldots,f_{m-1},\varphi,\varepsilon)\leq P_n(f_0,\ldots,f_{m-1},\varphi,\varepsilon) \leq p_n(f_0,\ldots,f_{m-1},\varphi,\gamma)$.
\end{theorem}

\begin{proof}
According to Remark \ref{remark:separated sets1}(4), $Q_n(f_0,\ldots,f_{m-1},\varphi,\varepsilon)\leq P_n(f_0,\ldots,f_{m-1},\varphi,\varepsilon)$ for all $\varepsilon>0$.

$(\rmnum{1})$ For $w \in F_m^{+}$, let $F$ be a $(w,\frac{\delta}{2},f_0,\ldots,f_{m-1})$ spanning set then \[X=\bigcup_{x \in F}\bigcap_{w' \leq w}f_{w'}^{-1}\bar{B}(f_{w'}x; \frac{\delta}{2}).\] Since each $\bar{B}(f_{w'}x; \frac{\delta}{2})$ is a subset of a member of $\alpha$ we have $q_w(f_0,\ldots,f_{m-1},\varphi,\alpha) \leq \sum_{x \in F}e^{(S_wf)(x)} $, then $q_w(f_0,\ldots,f_{m-1},\varphi,\alpha) \leq Q_w(f_0,\ldots,f_{m-1},\varphi,\frac{\delta}{2})$ and \[q_n(f_0,\ldots,f_{m-1},\varphi,\alpha)\leq Q_n(f_0,\ldots,f_{m-1},\varphi,\frac{\delta}{2}).\]

$(\rmnum{2})$  For $w \in F_m^{+}$, let $E$ be a $(w,\varepsilon,f_0,\ldots,f_{m-1})$ separated subset of $X$. Note that $\gamma$ is an open cover with $diam(\gamma)\leq \varepsilon$ then no member of $\bigvee_{w'\leq w}f_{w'}^{-1}(\gamma)$ contains two elements of $E$, We have $\sum_{x \in E}e^{(S_wf)(x)}\leq p_w(f_0,\ldots,f_{m-1},\varphi,\gamma) $. Then $P_w(f_0,\ldots,f_{m-1},\varphi,\varepsilon) \leq p_w(f_0,\ldots,f_{m-1},\varphi,\gamma)$ and \[P_n(f_0,\ldots,f_{m-1},\varphi,\varepsilon) \leq p_n(f_0,\ldots,f_{m-1},\varphi,\gamma).\]
\end{proof}

\begin{remark}\label{remark:open cover1}
(1) If $\alpha, \gamma$ are open covers of $X$ and $\alpha<\gamma$ then $q_n(f_0,\ldots,f_{m-1},\varphi,\alpha)\leq q_n(f_0,\ldots,f_{m-1},\varphi,\gamma)$.

(2) If $\tau_{\alpha}=\sup\{|\varphi(x)-\varphi(y)| : d(x,y)\leq diam(\alpha)\}$, then $p_n(f_0,\ldots,f_{m-1},\varphi,\alpha)\leq e^{n\tau_{\alpha}}q_n(f_0,\ldots,f_{m-1},\varphi,\alpha)$.
\end{remark}

Next we show the existence of $\lim_{n\rightarrow \infty}\frac{1}{n}\log p_n(f_0,\ldots,f_{m-1},\varphi,\alpha)$. Before proving it we first give the following lemma.
\begin{lemma}\label{lem:impor}\cite{WALTER2}
Let $\{a_n\}_{n \geq 1}$ be a sequence of real numbers such that $a_{n+p}\leq a_n+a_p ~\forall n,p$ then $\lim_{n\rightarrow \infty}a_n/n$ exists and equals $\inf_n a_n/n$.(The limit could be $-\infty$ but if the $a_n$ are bounded below then the limit will be non-negative.)
\end{lemma}
\begin{lemma}\label{lem:equalinf}
If $\varphi \in C(X,\mathbb{R})$ and $\alpha$ is an open cover of $X$ then
\[\lim_{n\rightarrow \infty}\frac{1}{n}\log p_n(f_0,\ldots,f_{m-1},\varphi,\alpha)\]
exists and equals to $\inf_n(1/n)\log p_n(f_0,\ldots,f_{m-1},\varphi,\alpha)$.
\end{lemma}
\begin{proof}
For any $w=w^{(1)}w^{(2)},|w^{(1)}|=n_1,|w^{(2)}|=n_2,|w|=n_1+n_2$,
by Lemma \ref{lem:impor} it suffices to show \[p_{n_1+n_2}(f_0,\ldots,f_{m-1},\varphi,\alpha) \leq p_{n_1}(f_0,\ldots,f_{m-1},\varphi,\alpha)\cdot p_{n_2}(f_0,\ldots,f_{m-1},\varphi,\alpha).\]
If $\beta$ is a finite subcover of $\bigvee_{w'\leq w^{(1)}}f_{w'}^{-1}(\alpha)$ and $\gamma$ is a finite subcover of $\bigvee_{w'\leq w^{(2)}}f_{w'}^{-1}(\alpha)$ then $\beta \bigvee f_{w^{(1)}}^{-1}(\gamma)$ is a finite subcover of $\bigvee_{w'\leq w}f_{w'}^{-1}(\alpha)$, and we have
\[\sum_{D \in \beta \bigvee f_{w^{(1)}}^{-1}(\gamma)}\sup_{x \in D}e^{(S_w\varphi)}(x)\leq \left(\sum_{B \in \beta} \sup_{x \in B}e^{(S_{w^{(1)}}\varphi)}(x)\right)\cdot \left(\sum_{C \in \gamma} \sup_{x \in C}e^{(S_{w^{(2)}}\varphi)}(x)\right).\]
Therefore \[p_w(f_0,\ldots,f_{m-1},\varphi,\alpha) \leq p_{w^{(1)}}(f_0,\ldots,f_{m-1},\varphi,\alpha)\cdot p_{w^{(2)}}(f_0,\ldots,f_{m-1},\varphi,\alpha).\] Then
\begin{align*}
&\quad \frac{1}{m^{n_1+n_2}}\sum_{|w|=n_1+n_2}p_w(f_0,\ldots,f_{m-1},\varphi,\alpha)\\ & \leq \frac{1}{m^{n_1+n_2}}\sum_{|w|=n_1+n_2}p_{w^{(1)}}(f_0,\ldots,f_{m-1},\varphi,\alpha)\cdot p_{w^{(2)}}(f_0,\ldots,f_{m-1},\varphi,\alpha)\\ &=\left(\frac{1}{m^{n_1}}\sum_{|w^{(1)}|=n_1}p_{w^{(1)}}(f_0,\ldots,f_{m-1},\varphi,\alpha)\right)\cdot \left(\frac{1}{m^{n_2}}\sum_{|w^{(2)}|=n_2}p_{w^{(2)}}(f_0,\ldots,f_{m-1},\varphi,\alpha)\right).
\end{align*}
Thus  \[p_{n_1+n_2}(f_0,\ldots,f_{m-1},\varphi,\alpha) \leq p_{n_1}(f_0,\ldots,f_{m-1},\varphi,\alpha)\cdot p_{n_2}(f_0,\ldots,f_{m-1},\varphi,\alpha).\]

\end{proof}

The following theorem gives equivalent definitions of the topological pressure using open covers. The proof of Theorem \ref{thm:open cover2} is analogous to the classical case but for convenience of  a reader a modified proof is given.

\begin{theorem}\label{thm:open cover2}
Let $f_i:X\rightarrow X$ be continuous, $i=0,\ldots,m-1$, and $\varphi \in C(X,\mathbb{R})$. Then each of the following equals $P(f_0,\ldots,f_{m-1},\varphi)$.

$(\rmnum{1})$ $\lim_{\delta\rightarrow 0}[\sup_{\alpha}\{\lim_{n \rightarrow \infty}\frac{1}{n}\log p_n(f_0,\ldots,f_{m-1},\varphi,\alpha) | \alpha~is~an ~open ~cover~ \\ of~ X ~with~ diam(\alpha) \leq \delta \}]$.

$(\rmnum{2})$ $\lim_{k\rightarrow \infty}[\lim_{n \rightarrow \infty}\frac{1}{n}\log p_n(f_0,\ldots,f_{m-1},\varphi,\alpha_k)]$, where $ \{\alpha_k\}~is~a \\ ~sequence ~of ~open~covers~ with~ diam(\alpha_k) \rightarrow 0 $.

$(\rmnum{3})$ $\lim_{\delta\rightarrow 0}[\sup_{\alpha}\{\liminf_{n \rightarrow \infty}\frac{1}{n}\log q_n(f_0,\ldots,f_{m-1},\varphi,\alpha) | \alpha~is~an ~open ~cover~ \\ of~ X ~with~ diam(\alpha) \leq \delta \}]$.

$(\rmnum{4})$ $\lim_{\delta\rightarrow 0}[\sup_{\alpha}\{\limsup_{n \rightarrow \infty}\frac{1}{n}\log q_n(f_0,\ldots,f_{m-1},\varphi,\alpha) | \alpha~is~an ~open ~cover~ \\ of~ X ~with~ diam(\alpha) \leq \delta \}]$.

$(\rmnum{5})$ $\lim_{k\rightarrow \infty}[\limsup_{n \rightarrow \infty}\frac{1}{n}\log q_n(f_0,\ldots,f_{m-1},\varphi,\alpha_k)]$, where $ \{\alpha_k\}~is~a \\ ~sequence ~of ~open~covers~ with~ diam(\alpha_k) \rightarrow 0 $.

$(\rmnum{6})$ $\sup_{\alpha}\ \{\limsup_{n \rightarrow \infty}\frac{1}{n}\log q_n(f_0,\ldots,f_{m-1},\varphi,\alpha) | \alpha~is~an ~open ~cover~  of~ X  \}$.

$(\rmnum{7})$ $\lim_{\varepsilon\rightarrow 0} \liminf_{n \rightarrow \infty}\frac{1}{n}\log Q_n(f_0,\ldots,f_{m-1},\varphi,\varepsilon)$.

$(\rmnum{8})$ $\lim_{\varepsilon\rightarrow 0} \liminf_{n \rightarrow \infty}\frac{1}{n}\log P_n(f_0,\ldots,f_{m-1},\varphi,\varepsilon)$.

\end{theorem}

\begin{proof}
$(\rmnum{1})$ If $\delta >0$ and $\alpha$ is an open cover with $diam(\alpha) \leq \delta$, according to Theorem\ref{thm:open cover1}$(\rmnum{2})$, we have $P_n(f_0,\ldots,f_{m-1},\varphi,\delta) \leq p_n(f_0,\ldots,f_{m-1},\varphi,\alpha)$. Using Lemma \ref{lem:equalinf}, we have
\begin{align*}
& \quad P(f_0,\ldots,f_{m-1},\varphi,\delta)\\ &= \limsup_{n \rightarrow \infty}\frac{1}{n} \log P_n(f_0,\ldots,f_{m-1},\varphi,\delta)\\ &\leq \limsup_{n \rightarrow \infty}\frac{1}{n} \log p_n(f_0,\ldots,f_{m-1},\varphi,\alpha)\\ &= \lim_{n \rightarrow \infty}\frac{1}{n} \log p_n(f_0,\ldots,f_{m-1},\varphi,\alpha) \\ & \leq \sup_{\alpha}\left\{\lim_{n \rightarrow \infty}\frac{1}{n}\log p_n(f_0,\ldots,f_{m-1},\varphi,\alpha)|~\alpha~is~an~open~cover~of~X~with~diam(\alpha) \leq \delta\right\}.
\end{align*}

Now, taking $\limsup$ when $\delta \rightarrow 0$ gives that $P(f_0,\ldots,f_{m-1},\varphi)$ is no lager than the expression in $(\rmnum{1})$.

If $\alpha$ is an open cover with Lebesgue number $\delta$, then $q_n(f_0,\ldots,f_{m-1},\varphi,\alpha)\leq Q_n(f_0,\ldots,f_{m-1},\varphi,\frac{\delta}{2})$ by Theorem \ref{thm:open cover1}($\rmnum{1}$). Remark \ref{remark:open cover1}(2) shows that $ \\ p_n(f_0,\ldots,f_{m-1},\varphi,\alpha)\leq e^{n\tau_{\alpha}}q_n(f_0,\ldots,f_{m-1},\varphi,\alpha)$, where $\tau_{\alpha}=\sup\{|\varphi(x)-\varphi(y)| : d(x,y)\leq diam(\alpha)\}$. Thus \[p_n(f_0,\ldots,f_{m-1},\varphi,\alpha) \leq  e^{n\tau_{\alpha}}Q_n(f_0,\ldots,f_{m-1},\varphi,\frac{\delta}{2})\]
This implies
\begin{align*}
\lim_{n \rightarrow \infty}\frac{1}{n} \log p_n(f_0,\ldots,f_{m-1},\varphi,\alpha) & \leq \tau_{\alpha}+ P(f_0,\ldots,f_{m-1},\varphi,\frac{\delta}{2})
\end{align*}
and \[\lim_{\eta\rightarrow 0}\left[\sup_{\alpha}\{\lim_{n \rightarrow \infty}\frac{1}{n}\log p_n(f_0,\ldots,f_{m-1},\varphi,\alpha) | diam(\alpha) \leq \eta \}\right] \leq P(f_0,\ldots,f_{m-1},\varphi).\]
Therefore $(\rmnum{1})$ is proved. The same reasoning proves $(\rmnum{2})$.

$(\rmnum{3})$ We know $q_n(f_0,\ldots,f_{m-1},\varphi,\alpha)\leq p_n(f_0,\ldots,f_{m-1},\varphi,\alpha)$ for all $\alpha$. Remark \ref{remark:open cover1}(2) shows that $  p_n(f_0,\ldots,f_{m-1},\varphi,\alpha)\leq e^{n\tau_{\alpha}}q_n(f_0,\ldots,f_{m-1},\varphi,\alpha)$, where $\tau_{\alpha}=\sup\{|\varphi(x)-\varphi(y)| : d(x,y)\leq diam(\alpha)\}$. Hence \[e^{-n\tau_{\alpha}}p_n(f_0,\ldots,f_{m-1},\varphi,\alpha)\leq q_n(f_0,\ldots,f_{m-1},\varphi,\alpha)\leq p_n(f_0,\ldots,f_{m-1},\varphi,\alpha).\] This implies
\begin{align*}
-\tau_{\alpha}+\lim_{n \rightarrow \infty}\frac{1}{n} \log p_n(f_0,\ldots,f_{m-1},\varphi,\alpha) & \leq \liminf_{n \rightarrow \infty}\frac{1}{n} \log q_n(f_0,\ldots,f_{m-1},\varphi,\alpha) \\ & \leq \limsup_{n \rightarrow \infty}\frac{1}{n} \log q_n(f_0,\ldots,f_{m-1},\varphi,\alpha) \\ & \leq \lim_{n \rightarrow \infty}\frac{1}{n} \log p_n(f_0,\ldots,f_{m-1},\varphi,\alpha).
\end{align*}
The formulae in $(\rmnum{3}),(\rmnum{4})$, and $(\rmnum{5})$ follow from $(\rmnum{1})$ and $(\rmnum{2})$.

$(\rmnum{6})$ Let $\alpha$ be an open cover of $X$ and let $2\varepsilon$ be a Lebesgue number for $\alpha$. By Theorem \ref{thm:open cover1} $q_n(f_0,\ldots,f_{m-1},\varphi,\alpha)\leq Q_n(f_0,\ldots,f_{m-1},\varphi,\varepsilon)$ so that \[\limsup_{n\rightarrow \infty}q_n(f_0,\ldots,f_{m-1},\varphi,\alpha) \leq Q(f_0,\ldots,f_{m-1},\varphi,\varepsilon)\leq P(f_0,\cdots,f_{m-1},\varphi).\] Therefore the expression in $(\rmnum{6})$ is majorised by $P(f_0,\cdots,f_{m-1},\varphi)$. The opposite inequality follows from $(\rmnum{4})$.

$(\rmnum{7})$ and $(\rmnum{8})$ Let $\alpha_{\varepsilon}$ be an open cover of $X$ by all open balls of radius $2\varepsilon$ and $\gamma_{\varepsilon}$ denote any cover by balls of radius $\varepsilon/2$. Then, by Theorem \ref{thm:open cover1}$(\rmnum{1})$ and Remark \ref{remark:open cover1}(2) we have
\begin{align*}
&\quad e^{-n\tau_{4\varepsilon}}p_n(f_0,\ldots,f_{m-1},\varphi,\alpha_{\varepsilon})\\ &\leq q_n(f_0,\ldots,f_{m-1},\varphi,\alpha_{\varepsilon}) \leq Q_n(f_0,\ldots,f_{m-1},\varphi,\varepsilon)  \leq P_n(f_0,\ldots,f_{m-1},\varphi,\varepsilon)\\ &\leq p_n(f_0,\ldots,f_{m-1},\varphi,\gamma_{\varepsilon}),
\end{align*}
where $\tau_{4\varepsilon}=\sup\{|\varphi(x)-\varphi(y)|:d(x,y)\leq 4\varepsilon\}$. Then $(\rmnum{7})$ and $(\rmnum{8})$ follow by taking lim infs in this expression and using $(\rmnum{2})$.
\end{proof}
\begin{remark}\label{remark:open cover2}
According to $(\rmnum{6})$ of Theorem \ref{thm:open cover2}, $P(f_0,\ldots,f_{m-1},\varphi)$ does not depend on the metric on $X$.
\end{remark}
\subsection{Properties of the topological Pressure}\label{subsec:properties of pressure}
We now study the properties of $P(f_0,\ldots,f_{m-1},\cdot):C(X,\mathbb{R})\rightarrow \mathbb{R} \cup \{\infty\}$. The proof of Theorem \ref{thm:properties of pressure1} is analogous to the classical case but for convenience of a reader a modified proof is given.
\begin{theorem}\label{thm:properties of pressure1}
Let $f_i:X\rightarrow X$ be continuous transformations of a compact metric space $(X,d)$, $i=0,\ldots,m-1$. If $\varphi, \psi \in C(X,\mathbb{R}),\varepsilon>0$ and $c \in \mathbb{R}$, then  the following are true.

$(\rmnum{1})$ $P(f_0,\ldots,f_{m-1},0)=h(f_0,\ldots,f_{m-1})$.

$(\rmnum{2})$ $\varphi\leq \psi$ implies $P(f_0,\ldots,f_{m-1},\varphi)\leq P(f_0,\ldots,f_{m-1},\psi)$. In particular, $h(f_0,\ldots,f_{m-1})+ \inf \varphi \leq P(f_0,\ldots,f_{m-1},\varphi) \leq h(f_0,\ldots,f_{m-1})+\sup \varphi $.

$(\rmnum{3})$ $P(f_0,\ldots,f_{m-1},\cdot)$ is either finite valued or constantly $\infty$.

$(\rmnum{4})$ $|P(f_0,\ldots,f_{m-1},\varphi,\varepsilon)-P(f_0,\ldots,f_{m-1},\psi,\varepsilon)| \leq \|\varphi-\psi\|$, and so if

\noindent $P(f_0,\ldots,f_{m-1},\cdot) < \infty$, then $|P(f_0,\ldots,f_{m-1},\varphi)-P(f_0,\ldots,f_{m-1},\psi| \leq \|\varphi-\psi\|$. In other words, $P(f_0,\ldots,f_{m-1},\cdot)$ is a continuous function on $C(X, \mathbb{R})$.

$(\rmnum{5})$ $P(f_0,\ldots,f_{m-1},\cdot,\varepsilon)$ is convex, and so if $P(f_0,\ldots,f_{m-1},\cdot) < \infty$, then $P(f_0,\ldots,f_{m-1},\cdot)$ is convex.

$(\rmnum{6})$ $P(f_0,\ldots,f_{m-1},\varphi+c)=P(f_0,\ldots,f_{m-1},\varphi)+c$.

$(\rmnum{7})$ $P(f_0,\ldots,f_{m-1},\varphi+\psi) \leq P(f_0,\ldots,f_{m-1},\varphi)+P(f_0,\ldots,f_{m-1},\psi)+ \log m$.

$(\rmnum{8})$ $P(f_0,\ldots,f_{m-1},c\varphi)\leq cP(f_0,\ldots,f_{m-1},\varphi)+(c-1)\log m$ if $c \geq 1$ and $P(f_0,\ldots,f_{m-1},c\varphi)\geq cP(f_0,\ldots,f_{m-1},\varphi)+(c-1)\log m$ if $c \leq 1$.

$(\rmnum{9})$ $-2\log m-P(f_0,\ldots,f_{m-1},|\varphi|) \leq P(f_0,\ldots,f_{m-1},\varphi) \leq P(f_0,\ldots,f_{m-1},|\varphi|)$.
\end{theorem}

\begin{proof}

$(\rmnum{1})$ and $(\rmnum{2})$ are clear from the definition of pressure.

$(\rmnum{3})$  According to $(\rmnum{1})$ and $(\rmnum{2})$, we know $h(f_0,\ldots,f_{m-1})+\inf \varphi\leq P(f_0,\ldots,f_{m-1},\varphi)\leq h(f_0,\ldots,f_{m-1})+\sup \varphi$. Note that $\inf \varphi, \sup \varphi$ are finite since $X$ is compact and $\varphi \in C(X, \mathbb{R})$. Thus, $P(f_0,\ldots,f_{m-1},\varphi)=\infty$ if and only if $h(f_0,\ldots,f_{m-1})=\infty$.

$(\rmnum{4})$ By the inequality $\sup_E ab \leq \sup_E a\sup_E b$, $a,b \geq 0$, for $w \in F_m^{+}$, we obtain

\begin{align*}
 &\quad \frac{P_w(f_0,\ldots,f_{m-1},\varphi,\varepsilon)}{P_w(f_0,\ldots,f_{m-1},\psi,\varepsilon)}\\& = \frac{  \sup \{\sum_{x \in E} e^{(S_w \varphi)(x)} | E ~is~ a~ (w,\varepsilon,f_0,\ldots,f_{m-1})~separated~ subset~ of~ X \}}{\sup\{\sum_{x \in E}e^{(S_w\psi)(x)}| E ~is~ a~ (w,\varepsilon,f_0,\ldots,f_{m-1}) ~separated~ subset ~of~ X \}} \\ & \leq \sup \left \{\frac{\sum_{x \in E} e^{(S_w \varphi)(x)}}{\sum_{x \in E}e^{(S_w\psi)(x)}} \Big | E ~is~ a~ (w,\varepsilon,f_0,\ldots,f_{m-1})~separated~ subset~ of~ X \right \} \\ & \leq \sup \left \{\max_{x \in E} \frac{\sum_{x \in E} e^{(S_w \varphi)(x)}}{\sum_{x \in E}e^{(S_w\psi)(x)}} \Big | E ~is~ a~ (w,\varepsilon,f_0,\ldots,f_{m-1})~separated~ subset~ of~ X \right \} \\ & \leq e^{|w|\|\varphi-\psi\|}.
\end{align*}
 Hence \[P_w(f_0,\ldots,f_{m-1},\varphi,\varepsilon) \leq e^{|w|\|\varphi-\psi\|}P_w(f_0,\ldots,f_{m-1},\psi,\varepsilon).\] Moreover \[P_n(f_0,\ldots,f_{m-1},\varphi,\varepsilon) \leq e^{n\|\varphi-\psi\|}P_n(f_0,\ldots,f_{m-1},\psi,\varepsilon).\]
 This implies \[P(f_0,\ldots,f_{m-1},\varphi,\varepsilon)-P(f_0,\ldots,f_{m-1},\psi,\varepsilon) \leq  \|\varphi-\psi\|.\]
 By the same way, we can show that \[P(f_0,\ldots,f_{m-1},\psi,\varepsilon)-P(f_0,\ldots,f_{m-1},\varphi,\varepsilon) \leq  \|\varphi-\psi\|.\]
 Therefore, \[|P(f_0,\ldots,f_{m-1},\varphi,\varepsilon)-P(f_0,\ldots,f_{m-1},\psi,\varepsilon)| \leq  \|\varphi-\psi\|.\]
 If $P(f_0,\ldots,f_{m-1},\cdot)< \infty$, letting $\varepsilon \rightarrow 0$, we obtain \[|P(f_0,\ldots,f_{m-1},\varphi)-P(f_0,\ldots,f_{m-1},\psi)| \leq  \|\varphi-\psi\|.\]

 $(\rmnum{5})$ For $w \in F_m^{+}$,  by H$\ddot{o}$lder inequality, if $p \in [0,1]$ and $E$ is a finite subset of $X$, we have \[\sum_{x \in E}e^{p(S_w\varphi)(x)+(1-p)(S_w\psi)(x)} \leq \left(\sum_{x \in E}e^{(S_w\varphi)(x)}\right)^p \left(\sum_{x \in E}e^{(S_w\psi)(x)}\right)^{1-p}.\]
 Thus
 \begin{align*}
&\quad P_w(f_0,\ldots,f_{m-1},p\varphi+(1-p)\psi,\varepsilon) \\& \leq (P_w(f_0,\ldots,f_{m-1},\varphi,\varepsilon)^p(P_w(f_0,\ldots,f_{m-1},\psi,\varepsilon))^{1-p}.
\end{align*}
By H$\ddot{o}$lder inequality again
 \begin{align*}
& \quad \sum_{|w|=n}P_w(f_0,\ldots,f_{m-1},p\varphi+(1-p)\psi,\varepsilon) \\ & \leq \sum_{|w|=n}(P_w(f_0,\ldots,f_{m-1},\varphi,\varepsilon)^p(P_w(f_0,\ldots,f_{m-1},\psi,\varepsilon))^{1-p} \\ & \leq \left(\sum_{|w|=n} P_w(f_0,\ldots,f_{m-1},\varphi,\varepsilon)\right)^p \left(\sum_{|w|=n} P_w(f_0,\ldots,f_{m-1},\psi,\varepsilon)\right)^{1-p}.
\end{align*}
Hence
 \begin{align*}
&\quad \frac{1}{m^n} \sum_{|w|=n}P_w(f_0,\ldots,f_{m-1},p\varphi+(1-p)\psi,\varepsilon) \\ & \leq \left(\frac{1}{m^n}\sum_{|w|=n} P_w(f_0,\ldots,f_{m-1},\varphi,\varepsilon)\right)^p \left(\frac{1}{m^n}\sum_{|w|=n} P_w(f_0,\ldots,f_{m-1},\psi,\varepsilon)\right)^{1-p}.
\end{align*}
Thus
 \begin{align*}
&\quad P_n(f_0,\ldots,f_{m-1},p\varphi+(1-p)\psi,\varepsilon) \\ & \leq (P_n(f_0,\ldots,f_{m-1},\varphi,\varepsilon))^p(P_n(f_0,\ldots,f_{m-1},\psi,\varepsilon))^{1-p}.
\end{align*}
Then we obtain
 \begin{align*}
&\quad P(f_0,\ldots,f_{m-1},p\varphi+(1-p)\psi,\varepsilon) \\ & \leq pP(f_0,\ldots,f_{m-1},\varphi,\varepsilon)+(1-p)(P_n(f_0,\ldots,f_{m-1},\psi,\varepsilon).
\end{align*}
If $P(f_0,\ldots,f_{m-1},\cdot)< \infty$, letting $\varepsilon \rightarrow 0$, then
 \begin{align*}
&\quad P(f_0,\ldots,f_{m-1},p\varphi+(1-p)\psi) \\& \leq pP(f_0,\ldots,f_{m-1},\varphi)+(1-p)P_n(f_0,\ldots,f_{m-1},\psi).
\end{align*}

$(\rmnum{6})$ is clear from the definition of topological pressure.

$(\rmnum{7})$  For $w \in F_m^{+}$, note that \[P_w(f_0,\ldots,f_{m-1},\varphi+\psi,\varepsilon) \leq P_w(f_0,\ldots,f_{m-1},\varphi,\varepsilon)P_w(f_0,\ldots,f_{m-1},\psi,\varepsilon).\]
Thus
\begin{align*}
&\quad \frac{1}{m^n}\sum_{|w|=n}P_w(f_0,\ldots,f_{m-1},\varphi+\psi,\varepsilon) \\ & \leq \frac{1}{m^n}\sum_{|w|=n}P_w(f_0,\ldots,f_{m-1},\varphi,\varepsilon)P_w(f_0,\ldots,f_{m-1},\psi,\varepsilon) \\ & \leq m^n \left(\frac{1}{m^n}\sum_{|w|=n}P_w(f_0,\ldots,f_{m-1},\varphi,\varepsilon)\right)\left(\frac{1}{m^n}\sum_{|w|=n}P_w(f_0,\ldots,f_{m-1},\psi,\varepsilon)\right).
\end{align*}
Then
\begin{align*}
P_n(f_0,\ldots,f_{m-1},\varphi+\psi,\varepsilon)  & \leq m^nP_n(f_0,\ldots,f_{m-1},\varphi,\varepsilon)P_n(f_0,\ldots,f_{m-1},\psi,\varepsilon).
\end{align*}
Hence \[P(f_0,\ldots,f_{m-1},\varphi+\psi) \leq P(f_0,\ldots,f_{m-1},\varphi)+P(f_0,\ldots,f_{m-1},\psi)+ \log m.\]

$(\rmnum{8})$ If $a_1,\ldots,a_k$ are positive numbers with $\sum_{i=1}^ka_i=1$ then $\sum_{i=1}^ka_i^c \leq 1$ if $c \geq 1$, and $\sum_{i=1}^ka_i^c \geq 1$ if $c \leq 1$. Therefore, for $w \in F_m^{+} $, if $E$ is a $(w,\varepsilon,f_0,\ldots,f_{m-1})$ separated subset of $X$ we have
\[\sum_{x \in E}e^{c(S_w\varphi)(x)} \leq \left(\sum_{x \in E}e^{(S_w\varphi)(x)}\right)^c~~ if ~c \geq 1\]
and
\[\sum_{x \in E}e^{c(S_w\varphi)(x)} \geq \left(\sum_{x \in E}e^{(S_w\varphi)(x)}\right)^c~~ if ~c \leq 1.\]
Therefore
\[P_w(f_0,\ldots,f_{m-1},c\varphi,\varepsilon) \leq (P_w(f_0,\ldots,f_{m-1},\varphi,\varepsilon))^c~~if~c \geq 1\]
and
\[P_w(f_0,\ldots,f_{m-1},c\varphi,\varepsilon) \geq (P_w(f_0,\ldots,f_{m-1},\varphi,\varepsilon))^c~~if~c \leq 1.\]
\textbf{Case 1}, if $c \geq 1$, then
\begin{align*}
&\quad \frac{1}{m^n}\sum_{|w|=n}P_w(f_0,\ldots,f_{m-1},c\varphi,\varepsilon) \\ & \leq \frac{1}{m^n}\sum_{|w|=n}(P_w(f_0,\ldots,f_{m-1},\varphi,\varepsilon))^c \\ &\leq \frac{1}{m^n}\left(\sum_{|w|=n}P_w(f_0,\ldots,f_{m-1},\varphi,\varepsilon)\right)^c \\ & =(\frac{1}{m^n})^{1-c}\left(\frac{1}{m^n}\sum_{|w|=n}(P_w(f_0,\ldots,f_{m-1},\varphi,\varepsilon))\right)^c.
\end{align*}
\textbf{Case 2}, if $c \leq 1$, then
\begin{align*}
&\quad \frac{1}{m^n}\sum_{|w|=n}P_w(f_0,\ldots,f_{m-1},c\varphi,\varepsilon)\\  & \geq \frac{1}{m^n}\sum_{|w|=n}(P_w(f_0,\ldots,f_{m-1},\varphi,\varepsilon))^c \\ &\geq \frac{1}{m^n}\left(\sum_{|w|=n}P_w(f_0,\ldots,f_{m-1},\varphi,\varepsilon)\right)^c \\ & =(\frac{1}{m^n})^{1-c}\left(\frac{1}{m^n}\sum_{|w|=n}(P_w(f_0,\ldots,f_{m-1},\varphi,\varepsilon))\right)^c
\end{align*}
 which implies the desired result.

$(\rmnum{9})$ Since $-|\varphi| \leq \varphi \leq |\varphi|$, by $(\rmnum{2})$, $P(f_0,\ldots,f_{m-1},-|\varphi|)\leq P(f_0,\ldots,f_{m-1},\varphi) \leq P(f_0,\ldots,f_{m-1},|\varphi|)$. From $(\rmnum{8})$ we have \[-P(f_0,\ldots,f_{m-1},|\varphi|)-2\log m \leq P(f_0,\ldots,f_{m-1},-|\varphi|).\]
Thus,
\[-2\log m-P(f_0,\ldots,f_{m-1},|\varphi|) \leq P(f_0,\ldots,f_{m-1},\varphi) \leq P(f_0,\ldots,f_{m-1},|\varphi|).\]
\end{proof}

\begin{remark}
If $m=1$, then the above results coincide with the results for the topological pressure of a single transformation\cite{WALTER2}.  If the free semigroup generated by $f_0,f_1,\ldots,f_{m-1}$ satisfies the so-called strongly $\delta^{*}$-expansive and $\varphi$ and $\psi$ satisfy the bounded distortion property, Rodrigues and Varandas\cite{RODRIGUES} proved the $(\rmnum{4})$ and $(\rmnum{6})$ in Theorem \ref{thm:properties of pressure1}.
\end{remark}

In order to  prove the partial variational principle, we give the following property of the topological pressure.

\begin{theorem}\label{prop:properties of pressure2}
Let $(X_i,d_i)$ be a compact metric space and let $G_i$ be the set of finite continuous transformations from $X_i$ into itself($i=1,2$), where $G_1=\{f_0,\ldots,f_{m-1}\}$ and $G_2=\{g_0,\ldots,g_{k-1}\}$. We use $G_1\times G_2$ denoting the semigroup acting  on the compact space $X_1\times X_2$ generated by $G_1\times G_2=\{f\times g: f \in G_1, g \in G_2\}:=\{(f \times g)_0,\cdots,(f \times g)_{mk-1}\}$ where $(f\times g)(x_1,x_2)=(f(x_1),g(x_2))$,$\forall f\times g \in G_1\times G_2$ and the metric on space $X_1\times X_2$ is given by $d((x_1,x_2),(y_1,y_2))=\max\{d_1(x_1,y_1),d_2(x_2,y_2)\}$. If $\varphi_i \in C(X_i, \mathbb{R})$, then
\[P(G_1\times G_2,\varphi_1\times \varphi_2) = P(G_1,\varphi_1)+P(G_2,\varphi_2)\]
where $\varphi_1 \times \varphi_2 \in C(X_1 \times X_2, \mathbb{R})$ is defined by $(\varphi_1\times \varphi_2)(x_1,x_2)=\varphi_1(x_1)+\varphi_2(x_2)$.
In particular, for any $n \in \mathbb{N}$, if $X_1=X_2=\cdots=X_n:=X$, then \[P(\overbrace{G\times G\times\cdots\times G}^{n},\overbrace{\varphi\times \varphi\times\cdots\times\varphi}^{n}) = nP(G,\varphi).\]
\end{theorem}

\begin{proof}
For any $w^{(1)}=w_1^{(1)}\cdots w_n^{(1)} \in F_m^{+}$ and  $w^{(2)}=w_1^{(2)}\cdots w_n^{(2)} \in F_k^{+}$, there exists only one $w=w_1\cdots w_n \in F_{mk}^{+}$ such that $(f\times g)_{w_i}=f_{w_i^{(1)}}\times g_{w_i^{(2)}}$. Define a map $h:F_m^{+}\times F_k^{+}\rightarrow F_{mk}^{+}$, such that $h(w^{(1)},w^{(2)})=w$. Then $h$ is a one-to-one correspondence. If $F_i$ is a $(w^{(i)},\varepsilon,f_0,\cdots,f_{m-1})$ spanning set for $X_i$ then $F_1 \times F_2$ is a $(w,\varepsilon,f_0,\cdots,f_{m-1})$ spanning set for $X_1 \times X_2$ with respect to $G_1 \times G_2$. Also

\begin{align*}
&\quad  \sum_{(x_1,x_2) \in F_1 \times F_2}\exp\left(\sum_{w' \leq w}(\varphi_1 \times \varphi_2)(f \times g)_{w'}(x_1,x_2) \right) \\ &= \left(\sum_{x_1 \in F_1}\exp\left(\sum_{w'_1 \leq w^{(1)}}\varphi_1((f)_{w'_1}(x_1)) \right)\right)\left(\sum_{x_2 \in F_2}\exp\left(\sum_{w'_2 \leq w^{(2)}}\varphi_2((g)_{w'_2}(x_2)) \right)\right)
\end{align*}
so that \[Q_w(G_1\times G_2,\varphi_1\times \varphi_2,\varepsilon) \leq Q_{w^{(1)}}(G_1,\varphi_1,\varepsilon)Q_{w^{(2)}}(G_1,\varphi_2,\varepsilon).\]
Then
\begin{align*}
&\quad \frac{1}{(mk)^n}\sum_{|w|=n}Q_w(G_1\times G_2,\varphi_1\times \varphi_2,\varepsilon)\\ & \leq \frac{1}{(mk)^n}\sum_{|w^{(1)}|=n,|w^{(2)}|=n}Q_{w^{(1)}}(G_1,\varphi_1,\varepsilon)Q_{w^{(2)}}(G_2,\varphi_2,\varepsilon) \\ & \leq \frac{1}{m^n}\sum_{|w^{(1)}|=n}Q_{w^{(1)}}(G_1,\varphi_1,\varepsilon)\frac{1}{k^n}\sum_{|w^{(2)}|=n}Q_{w^{(2)}}(G_2,\varphi_2,\varepsilon).
\end{align*}
which implies that \[Q_n(G_1\times G_2,\varphi_1\times \varphi_2,\varepsilon) \leq Q_n(G_1,\varphi_1,\varepsilon)\cdot Q_n(G_2,\varphi_2,\varepsilon).\]
Therefore  \[P(G_1\times G_2,\varphi_1\times \varphi_2) \leq P(G_1,\varphi_1)+P(G_2,\varphi_2).\]

If $E_i$ is a $(w^{(i)},\varepsilon,f_0,\cdots,f_{m-1})$ separated set for $X_i$ then $E_1 \times E_2$ is a

\noindent$(w,\varepsilon,f_0,\cdots,f_{m-1})$ separated set for $X_1 \times X_2$, so that
\[P_w(G_1\times G_2,\varphi_1\times \varphi_2,\varepsilon) \geq P_{w^{(1)}}(G_1,\varphi_1,\varepsilon)\cdot P_{w^{(2)}}(G_2,\varphi_2,\varepsilon).\]
Thus
\[P_n(G_1\times G_2,\varphi_1\times \varphi_2,\varepsilon) \geq P_n(G_1,\varphi_1,\varepsilon)\cdot P_n(G_2,\varphi_2,\varepsilon).\]
Since
\begin{align*}
& \quad \limsup_{n\rightarrow \infty}\frac{1}{n}\log P_n(G_1\times G_2,\varphi_1\times \varphi_2,\varepsilon)\\ & \geq \liminf_{n\rightarrow \infty}\frac{1}{n}\log P_n(G_1,\varphi_1,\varepsilon)+\limsup_{n\rightarrow \infty}\frac{1}{n}\log P_n(G_2,\varphi_2,\varepsilon).
\end{align*}
Theorem \ref{thm:open cover2}$(\rmnum{8})$ gives \[P(G_1\times G_2,\varphi_1\times \varphi_2) \geq P(G_1,\varphi_1)+P(G_2,\varphi_2).\]
\end{proof}

\section{measure-theoretic entropy for a free semigroup action}\label{sec:entropydefinition}
In this section, we introduce the measure-theoretic entropy of a free semigroup action and give some properties. The free semigroup is generated by finite measure-preserving transformations acting on a probability space.

First we give some notions. A partition of a probability space $(X,\mathscr{B},\mu)$ is a disjoint collection of elements of $\mathscr{B}$ whose union is $X$. Let $\xi=\{A_1,\cdots,A_k\}$ be a finite partition of $(X,\mathscr{B},\mu)$. Let $\eta=\{C_1,\cdots,C_l\}$ be another finite partition of $(X,\mathscr{B},\mu)$. The join of $\xi$ and $\eta$ is the partition \[\xi\vee \eta=\{A_i\cap C_j: 1\leq i \leq k,1 \leq j \leq l\}.\] We write $\xi\leq \eta$ to mean that each element of $\xi$ is a union of elements of $\eta$. Under the convention that $0\log0=0$, the entropy of the partition $\xi$ is \[H_{\mu}(\xi)=-\sum_{i=1}^k \mu(A_i) \log \mu(A_i).\] The conditional entropy of $\xi$ relative to $\eta$ is given by \[H_{\mu}(\xi|\eta)=-\sum_{\mu(C_j)\neq 0}\sum_{i=1}^k\mu(A_i\cap C_j)\log \frac{\mu(A_i\cap C_j)}{\mu(C_j)}.\] We denote the set of all finite partitions of $X$ by $\mathcal{L}$, then $\rho(\xi,\eta):=H_{\mu}(\xi|\eta)+H_{\mu}(\eta|\xi)$ is a metric on $\mathcal{L}$.

In order to give the definition of measure-theoretic entropy, we first prove the following lemma.

\begin{lemma}\label{lem:entropy1}
Let $f_0,\ldots,f_{m-1}$ be measure-preserving transformations of a probability space $(X,\mathscr{B},\mu)$, i.e., $\mu$ is $f_i$-invariant, $0\leq i\leq m-1$.  For $w \in F_m^{+}$, if $\xi \in \mathcal{L}$, then
\[\lim_{n\rightarrow \infty}\frac{1}{n}\left[\frac{1}{m^n}\sum_{|w|=n}H_{\mu}(\bigvee_{w' \leq w}f_{w'}^{-1}\xi)\right]\] exists.
\end{lemma}
\begin{proof}
For any $w=i_0\cdots i_{n_1-1}i_{n_1}\cdots i_{n_1+n_2-1}, w^{(1)}=i_0\cdots i_{n_1-1}, w^{(2)}=i_{n_1}\cdots i_{n_1+n_2-1}$, i.e., \[|w|=n_1+n_2,|w^{(1)}|=n_1,|w^{(2)}|=n_2.\]
By Lemma \ref{lem:impor} it suffices to show
\begin{align*}
&\quad \frac{1}{m^{n_1+n_2}}\sum_{|w|=n_1+n_2}H_{\mu}(\bigvee_{w' \leq w}f_{w'}^{-1}\xi)\\& \leq \frac{1}{m^{n_1}}\sum_{|w^{(1)}|=n_1}H_{\mu}(\bigvee_{w' \leq w^{(1)}}f_{w'}^{-1}\xi)+\frac{1}{m^{n_2}}\sum_{|w^{(2)}|=n_2}H_{\mu}(\bigvee_{w' \leq w^{(2)}}f_{w'}^{-1}\xi).
\end{align*}
Since $f_w=f_{i_0}\cdots f_{i_{n_1-1}}f_{i_{n_1}}\cdots f_{i_{n_1+n_2-1}},f_w^{-1}=f_{i_{n_1+n_2-1}}^{-1}\cdots f_{i_{n_1}}^{-1} f_{i_{n_1-1}}^{-1}\cdots f_{i_0}^{-1}$, then we have
\begin{align*}
H_{\mu}(\bigvee_{w' \leq w}f_{w'}^{-1}(\xi))  &=H_{\mu}(\bigvee_{w' \leq w^{(2)}}f_{w'}^{-1}\xi \vee f_{w^{(2)}}^{-1}(\bigvee_{w' \leq w^{(1)}}f_{w'}^{-1}\xi)) \\ & \leq H_{\mu}(\bigvee_{w' \leq w^{(2)}}f_{w'}^{-1}\xi)+ H_{\mu}(f_{w^{(2)}}^{-1}(\bigvee_{w' \leq w^{(1)}}f_{w'}^{-1}\xi)) \\ &=H_{\mu}(\bigvee_{w' \leq w^{(1)}}f_{w'}^{-1}\xi)+ H_{\mu}(\bigvee_{w' \leq w^{(2)}}f_{w'}^{-1}\xi).
\end{align*}
Therefore
\begin{align*}
&\quad \frac{1}{m^{n_1+n_2}}\sum_{|w|=n_1+n_2}H_{\mu}(\bigvee_{w' \leq w}f_{w'}^{-1}(\xi))  \\ &\leq \frac{1}{m^{n_1+n_2}}\sum_{|w|=n_1+n_2}\left(H_{\mu}(\bigvee_{w' \leq w^{(1)}}f_{w'}^{-1}\xi)+ H_{\mu}(\bigvee_{w' \leq w^{(2)}}f_{w'}^{-1}\xi)\right) \\ &=\frac{1}{m^{n_1+n_2}}\left(\sum_{|w^{(1)}|=n_1,|w^{(2)}|=n_2}H_{\mu}(\bigvee_{w' \leq w^{(1)}}f_{w'}^{-1}\xi)+ \sum_{|w^{(1)}|=n_1,|w^{(2)}|=n_2}H_{\mu}(\bigvee_{w' \leq w^{(2)}}f_{w'}^{-1}\xi)\right)\\ &=\frac{1}{m^{n_1+n_2}}\left(m^{n_2}\sum_{|w^{(1)}|=n_1}H_{\mu}(\bigvee_{w' \leq w^{(1)}}f_{w'}^{-1}\xi)+ m^{n_1}\sum_{|w^{(2)}|=n_2}H_{\mu}(\bigvee_{w' \leq w^{(2)}}f_{w'}^{-1}\xi)\right)\\ &=\frac{1}{m^{n_1}}\sum_{|w^{(1)}|=n_1}H_{\mu}(\bigvee_{w' \leq w^{(1)}}f_{w'}^{-1}\xi)+ \frac{1}{m^{n_2}}\sum_{|w^{(2)}|=n_2}H_{\mu}(\bigvee_{w' \leq w^{(2)}}f_{w'}^{-1}\xi).
\end{align*}
\end{proof}

\begin{definition}\label{def:entropy1}
Let $f_0,\ldots,f_{m-1}$ are measure-preserving transformations of a probability space $(X,\mathscr{B},\mu)$.  For $w \in F_m^{+}$, if $\xi \in \mathcal{L}$, denote
\[h_{\mu}(f_0,\ldots,f_{m-1}, \xi)=\lim_{n\rightarrow \infty}\frac{1}{n}\left[\frac{1}{m^n}\sum_{|w|=n}H_{\mu}(\bigvee_{w' \leq w}f_{w'}^{-1}\xi)\right].\]
The measure-theoretic entropy for a free semigroup action is defined by \[h_{\mu}(f_0,\ldots,f_{m-1})=\sup_{\xi \in \mathcal{L}} h_{\mu}(f_0,\ldots,f_{m-1},\xi).\] If we let $G:=\{f_0,\cdots,f_{m-1}\}$, then we also denote $h_{\mu}(f_0,\ldots,f_{m-1})$ by $h_{\mu}(G)$.
\end{definition}

\begin{remark}
(1) Obviously, $h_{\mu}(f_0,\ldots,f_{m-1}) \geq 0$.

(2) If $m=1$, then $h_{\mu}(f_0)$ is the classical measure-theoretic entropy of a single transformation(see, e.g. \cite{WALTER2}).
\end{remark}

In general, the basic properties of the measure-theoretic entropy of finite measure-preserving transformations are summarized below, which are similar to that of the classical measure-theoretic entropy(see e.g. \cite{WALTER2},~chap.4).

\begin{theorem}\label{thm:entropy1}
Let $(X,\mathscr{B},\mu)$ be a probability space and $f_0,\ldots,f_{m-1}$ measure-preserving transformations on $X$ preserving $\mu$. For $w \in F_m^{+}$, if $\xi, \eta \in \mathcal{L}$, then

$(\rmnum{1})$ $h_{\mu}(f_0,\ldots,f_{m-1},\xi) \leq H_{\mu}(\xi)$.

$(\rmnum{2})$ $h_{\mu}(f_0,\ldots,f_{m-1},\xi\vee \eta) \leq h_{\mu}(f_0,\ldots,f_{m-1},\xi)+h_{\mu}(f_0,\ldots,f_{m-1},\eta)$.

$(\rmnum{3})$ If $\xi \leq \eta$, then $h_{\mu}(f_0,\ldots,f_{m-1},\xi) \leq h_{\mu}(f_0,\ldots,f_{m-1},\eta)$.

$(\rmnum{4})$ $h_{\mu}(f_0,\ldots,f_{m-1},\xi) \leq h_{\mu}(f_0,\ldots,f_{m-1},\eta)+H_{\mu}(\xi | \eta)$.

$(\rmnum{5})$ $|h_{\mu}(f_0,\ldots,f_{m-1},\xi)-h_{\mu}(f_0,\ldots,f_{m-1},\eta)| \leq \rho(\xi,\eta)$. Hence the map $h_{\mu}(f_0,\ldots,f_{m-1},\cdot):\mathcal{L}\rightarrow \mathbb{R^{+}}\cup \{0\}$ is continuous.
\end{theorem}

\begin{proof}
We can get the desired results easily from the following facts respectively.
$(\rmnum{1})$
\begin{align*}
 \frac{1}{n}\left[\frac{1}{m^n}\sum_{|w|=n}H_{\mu}(\bigvee_{w' \leq w}f_{w'}^{-1}\xi)\right] & \leq \frac{1}{n}\left[\frac{1}{m^n}\sum_{|w|=n}\sum_{w' \leq w}H_{\mu}(f_{w'}^{-1}\xi)\right] \\ & =\frac{1}{n}\left[\frac{1}{m^n}\sum_{|w|=n}\sum_{w' \leq w}H_{\mu}(\xi)\right] \\ &=H_{\mu}(\xi).
\end{align*}
$(\rmnum{2})$
\begin{align*}
& \quad \frac{1}{m^n}\sum_{|w|=n}H_{\mu}(\bigvee_{w' \leq w}f_{w'}^{-1}(\xi \vee \eta)) \\ &=\frac{1}{m^n}\sum_{|w|=n}H_{\mu}(\bigvee_{w' \leq w}f_{w'}^{-1}\xi \vee \bigvee_{w' \leq w}f_{w'}^{-1}\eta) \\ & \leq \frac{1}{m^n}\sum_{|w|=n}H_{\mu}(\bigvee_{w' \leq w}f_{w'}^{-1}\xi) + \frac{1}{m^n}\sum_{|w|=n}H_{\mu}(\bigvee_{w' \leq w}f_{w'}^{-1}\eta).
\end{align*}
$(\rmnum{3})$ If $\xi \leq \eta$, then for any $w \in F_m^{+}$, \[\bigvee_{w' \leq w}f_{w'}^{-1}\xi \leq \bigvee_{w' \leq w}f_{w'}^{-1}\eta.\] Then \[H_{\mu}(\bigvee_{w' \leq w}f_{w'}^{-1}\xi) \leq H_{\mu}(\bigvee_{w' \leq w}f_{w'}^{-1}\eta),\]
hence, \[\frac{1}{m^n}\sum_{|w|=n}H_{\mu}(\bigvee_{w' \leq w}f_{w'}^{-1}\xi) \leq \frac{1}{m^n}\sum_{|w|=n}H_{\mu}(\bigvee_{w' \leq w}f_{w'}^{-1}\eta).\]
$(\rmnum{4})$
\begin{align*}
& \quad \frac{1}{m^n}\sum_{|w|=n}H_{\mu}(\bigvee_{w' \leq w}f_{w'}^{-1}\xi) \\ & \leq \frac{1}{m^n}\sum_{|w|=n}H_{\mu}(\bigvee_{w' \leq w}f_{w'}^{-1}\xi \vee \bigvee_{w' \leq w}f_{w'}^{-1}\eta)\\ & = \frac{1}{m^n}\sum_{|w|=n}H_{\mu}(\bigvee_{w' \leq w}f_{w'}^{-1}\eta) + \frac{1}{m^n}\sum_{|w|=n}H_{\mu}(\bigvee_{w' \leq w}f_{w'}^{-1}\xi | \bigvee_{w' \leq w}f_{w'}^{-1}\eta).
\end{align*}
For any $w \in F_m^{+},|w|=n$,
\begin{align*}
H_{\mu}(\bigvee_{w' \leq w}f_{w'}^{-1}\xi | \bigvee_{w' \leq w}f_{w'}^{-1}\eta) & \leq
\sum_{w' \leq w}H_{\mu}(f_{w'}^{-1}\xi |\bigvee_{w' \leq w}f_{w'}^{-1}\eta) \\ & \leq
\sum_{w' \leq w}H_{\mu}(f_{w'}^{-1}\xi |f_{w'}^{-1}\eta) \\ &=nH_{\mu}(\xi|\eta).
\end{align*}
Thus \[\frac{1}{m^n}\sum_{|w|=n}H_{\mu}(\bigvee_{w' \leq w}f_{w'}^{-1}\xi) \leq \frac{1}{m^n}\sum_{|w|=n}H_{\mu}(\bigvee_{w' \leq w}f_{w'}^{-1}\eta)+nH_{\mu}(\xi | \eta).\]
$(\rmnum{5})$ By $(\rmnum{4})$
\begin{align*}
 |h_{\mu}(f_0,\ldots,f_{m-1},\xi)-h_{\mu}(f_0,\ldots,f_{m-1},\eta)| & \leq \max(H_{\mu}(\xi | \eta),H_{\mu}(\eta | \xi)) \\ & \leq \rho(\xi,\eta).
\end{align*}
\end{proof}

The following two technical lemmas, which will be useful in what follows, can be found in \cite{WALTER2}.
\begin{lemma}\label{lem:entropy1}
Let $r \geq 1$ be an integer. Then for all $\varepsilon>0$ there exists $\delta>0$ such that if $\xi=\{A_1,\ldots,A_r\}$ and $\eta=\{B_1,\ldots,B_r\}$ are two finite partitions satisfying the inequality $\sum_{i=1}^r\mu(A_i\triangle B_i)< \delta$, it holds that $\rho(\xi,\eta)<\varepsilon$.
\end{lemma}

\begin{lemma}\label{lem:entropy2}
Let $\mathscr{B}_0$ be an algebra such that the $\sigma$-algebra generated by $\mathscr{B}_0$, which is denoted $\sigma(\mathscr{B}_0)$, is $\mathscr{B}$. Let $\xi$ be a finite partition of $X$ containing elements from $\mathscr{B}$. Then, for all $\varepsilon>0$ there exists a finite partition $\eta$ containing elements from $\mathscr{B}_0$ and holding $\rho(\xi,\eta)<\varepsilon$.
\end{lemma}

\begin{theorem}\label{thm:entropy2}
Let $(X,\mathscr{B},\mu)$ be a probability space and $\mathscr{B}_0$  an algebra such that the $\sigma$-algebra generated by $\mathscr{B}_0$(denoted by $\sigma(\mathscr{B}_0)$) satisfies $\sigma(\mathscr{B}_0)=\mathscr{B}$. Let $\mathcal{L}_0$ be the set of finite partitions of $X$ containing elements from $\mathscr{B}_0$. Then \[h_{\mu}(f_0,\ldots,f_{m-1})=\sup\{h_{\mu}(f_0,\ldots,f_{m-1},\xi):\xi \in \mathcal{L}_0\}.\]
\end{theorem}

\begin{proof}
By Lemma \ref{lem:entropy2}, given any $\varepsilon>0$ and $\eta \in \mathcal{L}$ there exists a finite partition $\xi_{\varepsilon} \in \mathcal{L}_0$ such that $H_{\mu}(\eta|\xi_{\varepsilon})<\varepsilon$. Then by Theorem \ref{thm:entropy1}$(\rmnum{4})$ we have
\begin{align*}
 h_{\mu}(f_0,\ldots,f_{m-1},\eta) &\leq h_{\mu}(f_0,\ldots,f_{m-1},\xi_{\varepsilon})+H_{\mu}(\eta | \xi_{\varepsilon}) \\ &\leq h_{\mu}(f_0,\ldots,f_{m-1},\xi_{\varepsilon})+\varepsilon.
\end{align*}
So \[h_{\mu}(f_0,\ldots,f_{m-1},\eta) \leq \varepsilon+\sup\{h_{\mu}(f_0,\ldots,f_{m-1},\xi):\xi \in \mathcal{L}_0\}.\]
Since $\varepsilon$ was arbitrary it follows that \[h_{\mu}(f_0,\ldots,f_{m-1}) \leq \sup\{h_{\mu}(f_0,\ldots,f_{m-1},\xi):\xi \in \mathcal{L}_0\}.\]
The reverse inequality is obvious, and so the proof ends.
\end{proof}

In order to prove the partial variational principle, we give the following property of the measure-theoretic entropy.

\begin{theorem}\label{prop:entropy3}
Let $(X_1,\mathscr{B}_1,\mu_1)$ and $(X_2,\mathscr{B}_2,\mu_2)$ be two probability spaces and let $f_i:X_1\rightarrow X_1$, $g_j:X_2\rightarrow X_2$ be measure-preserving, where $f_i \in G_1:=\{f_0,\ldots,f_{m-1}\}$ and $g_j \in G_2:=\{g_0,\ldots,g_{k-1}\}$. Denote $G_1\times G_2=\{f\times g: f \in G_1, g \in G_2\}:=\{(f \times g)_0,\cdots,(f \times g)_{mk-1}\}$, where $(f\times g)(x_1,x_2)=(f(x_1),g(x_2))$,$\forall f\times g \in G_1\times G_2$. Then

$(\rmnum{1})$ $h_{\mu_1 \times \mu_2}(G_1\times G_2) \leq h_{\mu_1}(G_1)+h_{\mu_2}(G_2)$.

$(\rmnum{2})$ If $(X_i,\mathscr{B}_i,\mu_i)=(X,\mathscr{B},\mu),G_i=G,i=1,2$, then $h_{\mu \times \mu}(G\times G) = 2h_{\mu}(G)$. In particular, for any finite number $n \in \mathbb{N}$ \[h_{\tiny{\overbrace{(\mu \times \mu)\times\ldots\times(\mu \times \mu)}^{2^n}}}(\overbrace{(G\times G)\times\ldots\times(G\times G)}^{2^n}) = 2^{n+1}h_{\mu}(G).\]
\end{theorem}

\begin{proof}
$(\rmnum{1})$ Consider two measurable partitions, $\xi_1$ of $X_1$ and $\xi_2$ of $X_2$. Then $\xi_1\times \xi_2$ is a measurable partition of $X_1 \times X_2$. Hence, for $w^{(1)}=w_1^{(1)}\cdots w_n^{(1)} \in F_m^{+}$ and  $w^{(2)}=w_1^{(2)}\cdots w_n^{(2)} \in F_k^{+}$, there exists only one $w=w_1\cdots w_n \in F_{mk}^{+}$ such that
$(f\times g)_{w_i}=f_{w_i^{(1)}}\times g_{w_i^{(2)}}$. Hence
\begin{equation}\label{eq:rell}
H_{\mu_1 \times \mu_2}\left(\bigvee_{w' \leq w}(f\times g)_{w'}^{-1}(\xi_1 \times \xi_2)\right)=H_{\mu_1}(\bigvee_{w'_1 \leq w^{(1)}}f_{w'_1}^{-1}\xi_1)+H_{\mu_2}(\bigvee_{w'_2 \leq w^{(2)}}g_{w'_2}^{-1}\xi_2).
\end{equation}

Since the left side of the above formula has $(mk)^n$  terms for all $|w|=n$, where the two terms of the right hand side of the above formula have  $k^n$ and $m^n$ respectively  duplicates for every $|w^{(1)}|=|w^{(2)}|=n$, respectively, then we have
\begin{align*}
 &\quad h_{\mu_1 \times \mu_2}(G_1\times G_2, \xi_1 \times \xi_2) \\&=\lim_{n \rightarrow \infty}\frac{1}{n}\left[\frac{1}{(mk)^n}\sum_{|w|=n}H_{\mu_1 \times \mu_2}\left(\bigvee_{w' \leq w}(f\times g)_{w'}^{-1}(\xi_1 \times \xi_2)\right) \right]\\ &=\lim_{n \rightarrow \infty}\frac{1}{n}\left[\frac{1}{(mk)^n}\sum_{|w|=n}\left(H_{\mu_1}(\bigvee_{w'_1 \leq w^{(1)}}f_{w'_1}^{-1}\xi_1)+H_{\mu_2}(\bigvee_{w'_2 \leq w^{(2)}}g_{w'_2}^{-1}\xi_2) \right)\right]
  \\& = \lim_{n \rightarrow \infty}\frac{1}{n}\left[\frac{1}{(mk)^n}\left(\sum_{|w^{(1)}|=n,|w^{(2)}|=n}H_{\mu_1}(\bigvee_{w'_1 \leq w^{(1)}}f_{w'_1}^{-1}\xi_1)+\sum_{|w^{(1)}|=n,|w^{(2)}|=n}H_{\mu_2}(\bigvee_{w'_2 \leq w^{(2)}}g_{w'_2}^{-1}\xi_2) \right)\right]\\&= \lim_{n \rightarrow \infty}\frac{1}{n}\left[\frac{1}{(mk)^n}\left(k^n\sum_{|w^{(1)}|=n}H_{\mu_1}(\bigvee_{w'_1 \leq w^{(1)}}f_{w'_1}^{-1}\xi_1)+m^n\sum_{|w^{(2)}|=n}H_{\mu_2}(\bigvee_{w'_2 \leq w^{(2)}}g_{w'_2}^{-1}\xi_2) \right)\right] \\ & \leq h_{\mu_1}(G_1,\xi_1)+h_{\mu_2}(G_2,\xi_2) \\ &\leq h_{\mu_1}(G_1)+h_{\mu_2}(G_2).
\end{align*}

Since the algebra $\mathscr{B}_1 \times \mathscr{B}_2$ generates the product $\sigma$-algebra $\mathscr{B}$, by Theorem \ref{thm:entropy2}, it follows that
\begin{align*}
 h_{\mu_1 \times \mu_2}(G_1\times G_2) &\leq \sup_{\xi_1 \times \xi_2}h_{\mu_1 \times \mu_2}(G_1\times G_2, \xi_1 \times \xi_2)\\ &\leq h_{\mu_1}(G_1)+h_{\mu_2}(G_2).
\end{align*}
and the part $(\rmnum{1})$ follows.

$(\rmnum{2})$ Consider $\xi$ a finite partition of $X$ and From the formula (\ref{eq:rell}) we have
\begin{align*}
H_{\mu \times \mu}\left(\bigvee_{w' \leq w}(f\times g)_{w'}^{-1}(\xi \times \xi)\right)=2H_{\mu}\left(\bigvee_{w_1^{'} \leq w_1}f_{w_1^{'}}^{-1}\xi\right).
\end{align*}
Therefore
\begin{align*}
 &\quad h_{\mu \times \mu}(G\times G, \xi \times \xi) \\&=\lim_{n \rightarrow \infty}\frac{1}{n}\left[\frac{1}{m^{2n}}\sum_{|w|=n}H_{\mu \times \mu}\left(\bigvee_{w' \leq w}(f\times g)_{w'}^{-1}(\xi \times \xi)\right) \right]\\ & = \lim_{n \rightarrow \infty}\frac{1}{n}\left[\frac{2}{m^{2n}}\left(m^n\sum_{|w_1|=n}H_{\mu}(\bigvee_{w_1^{'} \leq w_1}f_{w_1^{'}}^{-1}\xi)\right)\right] \\ & =2h_{\mu}(G,\xi).
\end{align*}
Taking the supremum over all the finite partitions
\begin{align*}
 \sup_{\xi}h_{\mu \times \mu}(G\times G, \xi \times \xi)=2\sup_{\xi}h_{\mu}(G,\xi)=2 h_{\mu}(G).
\end{align*}
Then \[h_{\mu \times \mu}(G\times G) \geq 2 h_{\mu}(G),\] and using the part $(\rmnum{1})$ the proof ends.
\end{proof}

\section{The proofs of the main results}\label{sec:main results}
In this section, we give the proofs of Theorem \ref{thm:properties of pressure3} and Theorem \ref{thm:relation1}.

First, we link free semigroup actions and skew-product transformations by the similar way from Bufetov\cite{BUfE}.

Let $(X,d)$ be a compact metric space. Suppose that a free semigroup with $m$ generators acts on $X$; denote the maps corresponding to the generators by $f_0,\cdots,f_{m-1}$; we assume that these maps are continuous.

To this action, we assign the following skew-product transformation. Its base is $\Sigma_{m}$, its fiber is $X$, and the maps $F:\Sigma_{m}\times X \rightarrow \Sigma_{m}\times X$ and $g:\Sigma_{m}\times X \rightarrow \mathbb{R}$ are defined by the formula\[F(\omega,x)=(\sigma_{m}\omega,f_{\omega_0}(x))\] and \[g(\omega,x)=c+\varphi(x).\] Here $f_{\omega_0}$ stands for $f_0$ if $\omega_0=0$, and for $f_1$ if $\omega_0=1$, and so on;  $c$ is a constant number and $\varphi \in C(X, \mathbb{R})$. Obviously, $g \in C(\Sigma_{m}\times X, \mathbb{R})$. For $w=i_1\cdots i_k \in F_m^{+}$, denote $\overline{w}=i_k\cdots i_1$. Let $\omega=(\cdots,\omega_{-1},\omega_0,\omega_1,\cdots)\in \Sigma_m$, then
\begin{align*}
F^n(\omega,x)&=(\sigma_{m}^n\omega,f_{\omega_{n-1}}f_{\omega_{n-2}}\cdots f_{\omega_0}(x))\\&=(\sigma_{m}^n\omega,f_{\overline{\omega|_{[0,n-1]}}}(x)).
\end{align*}
Moreover, $S_ng(\omega,x)=nc+S_{\overline{\omega|_{[0,n-1]}}}\varphi(x)$.
Let the metric on $\Sigma_{m}\times X $ be given by the formula \[d((\omega,x),(\omega',x'))=\max(d(\omega,\omega'),d(x,x')).\]

Before proving Theorem \ref{thm:properties of pressure3}, we give the following two lemmas.

\begin{lemma}\label{lem:topological2}
For any natural $n$ and $0<\varepsilon<\frac{1}{2}$ \[P_n(F,g,\varepsilon) \geq e^{nc}m^nP_n(f_0,\cdots,f_{m-1},\varphi,\varepsilon).\]
\end{lemma}

\begin{proof}
Let $N=m^n$. There are $N$ distinct words of length $n$ in $F_m^{+}$. Denote these words by $w_1,\cdots,w_N$. For any $i=1,\cdots,N$, let $\omega(i) \in \Sigma_m$ be an arbitrary sequence such that $\omega(i)|_{[0,n-1]}=w_i$. Obviously, for $0<\varepsilon<\frac{1}{2}$, the sequences $\omega(i), i=1,\cdots,N$, form a $(n,\varepsilon,\sigma_m)$ separated subset of $\Sigma_m$.

Let $N_i=N(\overline{w}_i,\varepsilon,f_0,\cdots,f_{m-1})$, and let the points $x_1^i,\cdots,x_{N_i}^i$ form a $(\overline{w}_i,\varepsilon,f_0,\cdots,f_{m-1})$ separating subset of $X$. Then the points \[(\omega(i),x_j^i) \in \Sigma_{m}\times X, ~~ i=1,\cdots,N~~j=1,\cdots,N_i,\] form a $(n,\varepsilon,F)$ separating subset of $\Sigma_{m}\times X$. Hence we have
\begin{align*}
 P_n(F,g,\varepsilon) & \geq \sum_{(\omega,x)\in \{(\omega(i),x_j^i):i=1,\cdots,N~~j=1,\cdots,N_i\}}e^{S_ng(\omega,x)} \\ &= \sum_{(\omega,x)\in \{(\omega(i),x_j^i):i=1,\cdots,N~~j=1,\cdots,N_i\}}e^{nc+S_{\overline{w}_i}\varphi(x)} \\ & =e^{nc}\sum_{w_i \in \{w_1,\cdots,w_N\}}\sum_{x\in \{x_j^i:~~j=1,\cdots,N_i\}}e^{S_{\overline{w}_i}\varphi(x)} .
\end{align*}
Hence \[P_n(F,g,\varepsilon) \geq e^{nc}m^nP_n(f_0,\cdots,f_{m-1},\varphi,\varepsilon).\]
\end{proof}

\begin{lemma}\label{lem:topological3}
For any natural $n$ and $\varepsilon>0$ \[Q_n(F,g,\varepsilon) \leq K(\varepsilon)e^{nc}m^nQ_n(f_0,\cdots,f_{m-1},\varphi,\varepsilon).\] where $K(\varepsilon)$ is a positive constant that depends only on $\varepsilon$.
\end{lemma}

\begin{proof}
Let $C(\varepsilon)$ be an arbitrary positive integer such that $2^{-C(\varepsilon)}<\frac{\varepsilon}{100}$. Let $N=m^{n+2C(\varepsilon)}$. There are $N$ distinct words of length $n+2C(\varepsilon)$ in $F_m^{+}$. Denote these words by $w_1,\cdots,w_N$. For any $i=1,\cdots,N$, let $\omega(i) \in \Sigma_m$ be an arbitrary sequence such that $\omega(i)|_{[-C(\varepsilon),n+C(\varepsilon)-1]}=w_i$. Obviously, for $0<\varepsilon<\frac{1}{2}$, the sequences $\omega(i), i=1,\cdots,N$, form a $(n,\varepsilon,\sigma_m)$ spanning subset of $\Sigma_m$. Denote \[w'_i=\omega(i)|_{[0,n-1]},~~ B_i=B(\overline{w}'_i,\varepsilon,f_0,\cdots,f_{m-1})\] and assume that the points $x_1^i,\cdots,x_{B_i}^i$ form a $(\overline{w}'_i,\varepsilon,f_0,\cdots,f_{m-1})$ spanning subset of $X$. Then the points \[(\omega(i),x_j^i) \in \Sigma_m \times X, ~~ i=1,\cdots,N~~j=1,\cdots,B_i,\] form a $(n,\varepsilon,F)$ spanning subset of $\Sigma_{m}\times X$. Hence we have

\begin{align*}
 Q_n(F,g,\varepsilon) & \leq \sum_{(\omega,x)\in \{(\omega(i),x_j^i):i=1,\cdots,N~~j=1,\cdots,B_i\}}e^{S_ng(\omega,x)} \\ &\leq K(\varepsilon)\sum_{|w'_i|=n,~~ x \in \{x_j^i:j=1,\cdots,B_i\}}e^{nc+S_{\overline{w}'_i}\varphi(x)} \\ & =K(\varepsilon)e^{nc}\sum_{|w'_i|=n }\sum_{x\in \{x_j^i:~~j=1,\cdots,B_i\}}e^{S_{\overline{w}'_i}\varphi(x)}
\end{align*}
where $K(\varepsilon)$ is a positive constant that depends only on $\varepsilon$.
Hence \[Q_n(F,g,\varepsilon) \leq K(\varepsilon)e^{nc}m^nQ_n(f_0,\cdots,f_{m-1},\varphi,\varepsilon).\]

\end{proof}

\textbf{Proof of the Theorem \ref{thm:properties of pressure3}.}

From Lemma \ref{lem:topological2} we have \[P_n(F,g,\varepsilon) \geq e^{nc}m^nP_n(f_0,\cdots,f_{m-1},\varphi,\varepsilon), \] whence, taking logarithms and limits, we obtain that \[P(F,g) \geq c+\log m+P(f_0,\cdots,f_{m-1},\varphi).\]
In the same way, from Lemma \ref{lem:topological3}, we have \[Q_n(F,g,\varepsilon) \leq K(\varepsilon)e^{nc}m^nQ_n(f_0,\cdots,f_{m-1},\varphi,\varepsilon).\] whence \[P(F,g) \leq c+\log m+P(f_0,\cdots,f_{m-1},\varphi).\]
and the proof is complete. \qed

In the following, we prove Theorem \ref{thm:relation1}, i.e., the partial variational principle for a free semigroup action. First we give the definition of invariant measure for finite continuous maps.

Let $X$ be a compact space. Consider a finite continuous maps $f_0,\ldots,f_{m-1}$, denote by $\mathscr{M}(X)$ the set of all the probability measures on $(X, \mathscr{B}(X))$, where $\mathscr{B}(X)$ denotes the Borel $\sigma$-algebra of $X$. Then the set of invariant measures for $f_0,\ldots,f_{m-1}$, $\mathscr{M}(X,f_0,\ldots,f_{m-1})$ is the set of fixed points of the map $\tilde{f}_i: \mathscr{M}(X) \rightarrow \mathscr{M}(X)$ defined by $ \tilde{f}_i(\mu(A))=\mu(f_i^{-1}A)$ for $i=0,\ldots,m-1$. That is, $\mu \in \mathscr{M}(X,f_0,\ldots,f_{m-1})$ if and only if $\mu(f_i^{-1}A)=\mu(A)$ for all $A \in \mathscr{B}(X)$ and all $i=0,\ldots,m-1$. Let's remark that the set of invariant measures for the classical compact systems and the systems defined on the abelian semigroup $G$ are always non-empty. However, the following example\cite{CANOVAS} shows that $\mathscr{M}(X,f_0,\cdots,f_{m-1})$ can be empty.

\begin{example}
Let $f_0(x)=1$ for all $x \in X:=[0,1]$ and $f_1$ is the standard tent map $f_1(x)=1-|2x-1|$. The set of invariant measures $\mathscr{M}(X,f_0,f_1)$ is the set of fixed points of the maps $\tilde{f}_0$  and $\tilde{f}_1$ from $\mathscr{M}(X)$ into itself. $\tilde{f}_0$ has only a fixed point $\delta_0$, the probabilistic atomic measure such that $\delta_0(\{1\})=1$, but this measure is not a fixed point of $\tilde{f}_1$ and so $\mathscr{M}(X,f_0,f_1)$ is empty.
\end{example}

Moreover, we provide an example of nonabelian semigroup $H$ with $H$-invariant measure.

\begin{example}
Let $A$ and $B$ be the endomorphisms on 2-dimensional torus $\mathbb{T}^2$ induced by the matrices
\[
\left(\begin{array}{cccc}
    1 & 2 \\
    -1 & 4
\end{array}\right)
and
\left(\begin{array}{cccc}
    1 & -1 \\
    -1 & -3
\end{array}\right)
\]
respectively. Let $H$ be the semigroup generated by $A$ and $B$. Obviously, $H$ is a nonabelian semigroup. Let $\mu$ be the Haar measure defined on $\mathbb{T}^2$, then we have $\mu \in \mathscr{M}(\mathbb{T}^2,A,B)$, i.e., $\mathscr{M}(\mathbb{T}^2,A,B)\neq \emptyset$.
\end{example}

Hence we assume that the set of invariant measures $\mathscr{M}(X,f_0,\ldots,f_{m-1})$ is not empty(for instance, when $f_0,\ldots,f_{m-1}$ contains two commuting maps \cite{HU}). Studying the above mentioned relationship makes sense.
We start with the following lemma, which will be useful in what follows, can be found in \cite{WALTER2}.

\begin{lemma}\label{lem:relation1}
Let $a_1,\ldots,a_k$ be given real numbers. If $p_i \geq 0$ and $\sum_{i=1}^kp_i=1$ then \[\sum_{i=1}^kp_i(a_i-\log p_i)\leq \log(\sum_{i=1}^ke^{a_i}).\] and equality holds if and only if \[p_i=\frac{e^{a_i}}{\sum_{j=1}^ke^{a_j}}\]
\end{lemma}

\begin{lemma}\label{lem:relation2}
Let $(X,d)$ be a compact metric space and let $f_i:X\rightarrow X$(i=0,\ldots,m-1) be finite continuous maps such that $\mathscr{M}(X,f_0,\ldots,f_{m-1})\neq \emptyset$ and let $\varphi \in C(X,\mathbb{R})$. Then
 \[\sup \left\{h_{\mu}(f_0,\ldots,f_{m-1})+\int \varphi d\mu \Big| \mu \in \mathscr{M}(X,f_0,\ldots,f_{m-1})\right\} \leq \log 2+P(f_0,\ldots,f_{m-1},\varphi).\]
\end{lemma}

\begin{proof}
We use the analogous method as that of Misiurewicz\cite{MISI}. Let $\mu \in \mathscr{M}(X,f_0,\ldots,f_{m-1})$ and let $\xi=\{A_1,\ldots,A_k\}$ be a finite partition of $(X,\mathscr{B}(X))$. For any $a >0$, choose $\varepsilon>0$ so that $ \varepsilon k\log k<a$. Since $\mu$ is regular, there exist compact sets $B_j \subset A_j,1 \leq j \leq k$ with $\mu(A_j\backslash B_j) \leq \varepsilon$. Let $\eta$ be the partition $\eta=\{B_0,B_1,\ldots,B_k\}$ where $B_0=X\backslash \cup_{j=1}^kB_j$. We have $\mu(B_0)<k\varepsilon$ and
\begin{align*}
 H_{\mu}(\xi | \eta)&=-\sum_{i=0}^k\sum_{j=1}^k\mu(B_i)\frac{\mu(B_i \cap A_j)}{\mu(B_i)}\log \frac{\mu(B_i \cap A_j)}{\mu(B_i)} \\&=-\mu(B_0)\sum_{j=1}^k\frac{\mu(B_0 \cap A_j)}{\mu(B_0)}\log \frac{\mu(B_0 \cap A_j)}{\mu(B_0)} \\& \leq \mu(B_0)\log k \\ & <k\varepsilon\log k<a.
\end{align*}
Let \[b=\min_{1 \leq i\neq j \leq k}d(B_i,B_j)>0.\]
Pick $\delta>0$ so that $\delta<b/2$ and so that $d(x,y)<\delta$ implies $|\varphi(x)-\varphi(y)|<\varepsilon$. Fix $w \in F_m^{+}, |w|=n$ and let $E_w$ be an $(w,\delta,f_0,\ldots,f_{m-1})$ separated set, which fails to be $(w,\delta,f_0,\ldots,f_{m-1})$ separated when any point is added. Then $E_w$ is also $(w,\delta,f_0,\ldots,f_{m-1})$ spanning. If $C \in \bigvee_{w' \leq w}f_{w'}^{-1}\eta$ let $\alpha_w(C)$ denote $\sup\{(S_wf)(x)|x \in C\}$. Then
\begin{align*}
 H_{\mu}(\bigvee_{w' \leq w}f_{w'}^{-1}\eta)+\int S_w\varphi d\mu & \leq \sum_{C \in \bigvee_{w' \leq w}f_{w'}^{-1}\eta}\mu(C)[-\log \mu(C)+\alpha(C)] \\ & \leq \log \sum_{C \in \bigvee_{w' \leq w}f_{w'}^{-1}\eta}e^{\alpha(C)} ~by~ Lemma~ \ref{lem:relation1}.
\end{align*}
Since  $\frac{1}{m^n}\sum_{|w|=n}\int S_w\varphi d\mu=n\int \varphi d\mu$, then
\begin{equation}\label{eq:rel}
\begin{aligned}
& \quad \frac{1}{m^n}\sum_{|w|=n}H_{\mu}(\bigvee_{w' \leq w}f_{w'}^{-1}\eta)+\frac{1}{m^n}\sum_{|w|=n}\int S_w\varphi d\mu\\ &= \frac{1}{m^n}\sum_{|w|=n}H_{\mu}(\bigvee_{w' \leq w}f_{w'}^{-1}\eta)+n\int \varphi d\mu \\& \leq \frac{1}{m^n}\sum_{|w|=n}\log \sum_{C \in \bigvee_{w' \leq w}f_{w'}^{-1}\eta}e^{\alpha_w(C)}.
\end{aligned}
\end{equation}

For each $C \in \bigvee_{w' \leq w}f_{w'}^{-1}\eta$ choose some $x \in \bar{C}$ so that $(S_w\varphi)(x)=\alpha_w(C)$. Since $E_w$ is $(w,\delta,f_0,\ldots,f_{m-1})$ spanning choose $y(C)\in E_w$ with $d(f_{w'}x,f_{w'}y(C)) \leq \delta$, $w' \leq w$. Then $\alpha_w(C) \leq (S_w\varphi)(y(C))+n\varepsilon$. Also each ball of radius $\delta$ meets the closures of at most two members of $\eta$ so if $y \in E_w$ then  $\{C \in \bigvee_{w' \leq w}f_{w'}^{-1}\eta|y(C)=y\}$ has cardinality at most $2^n$. Therefore
\[\sum_{C \in \bigvee_{w' \leq w}f_{w'}^{-1}\eta}e^{\alpha_w(C)-n\varepsilon}\leq \sum_{C \in \bigvee_{w' \leq w}f_{w'}^{-1}\eta}e^{(S_w\varphi)(y(C))} \leq 2^n\sum_{y \in E_w}e^{(S_w\varphi)(y)}.\]
and so
\begin{equation}\label{equ:rel2}
\log\left(\sum_{C \in \bigvee_{w' \leq w}f_{w'}^{-1}\eta}e^{\alpha_w(C)}\right)-n\varepsilon \leq n\log 2+\log\left(\sum_{y \in E_w}e^{(S_w\varphi)(y)}\right).
\end{equation}
By the arithmetic-geometric mean inequality, we have
\[\frac{1}{m^n}\sum_{|w|=n}\log\left(\sum_{y \in E_w}e^{(S_w\varphi)(y)}\right) \leq \log\left(\frac{1}{m^n}\sum_{|w|=n}\sum_{y \in E_w}e^{(S_w\varphi)(y)}\right).\]
Combining with (\ref{eq:rel}) and (\ref{equ:rel2}), we have
\begin{align*}
&\quad \frac{1}{n}\frac{1}{m^n}\sum_{|w|=n}H_{\mu}(\bigvee_{w' \leq w}f_{w'}^{-1}\eta)+\int \varphi d\mu \\ &\leq \frac{1}{n}\frac{1}{m^n}\sum_{|w|=n}\log \left(\sum_{C \in \bigvee_{w' \leq w}f_{w'}^{-1}\eta}e^{\alpha_w(C)}\right) \\ & \leq \varepsilon+ \log2+\frac{1}{n}\frac{1}{m^n}\sum_{|w|=n}\log\left(\sum_{y \in E_w}e^{(S_w\varphi)(y)}\right) \\ &\leq \varepsilon+\log2+ \frac{1}{n}\log\left(\frac{1}{m^n}\sum_{|w|=n}\sum_{y \in E_w}e^{(S_w\varphi)(y)}\right) \\ & \leq \varepsilon+\log2+\frac{1}{n}\log\left(\frac{1}{m^n}\sum_{|w|=n}P_w(f_0,\ldots,f_{m-1},\varphi,\delta)\right) \\ &=\varepsilon+\log2+\frac{1}{n}\log P_n(f_0,\ldots,f_{m-1},\varphi,\delta).
\end{align*}
Hence
\begin{align*}
  h_{\mu}(f_0,\ldots,f_{m-1},\eta)+\int \varphi d\mu &\leq \varepsilon+\log2+P(f_0,\ldots,f_{m-1},\varphi,\delta) \\ & \leq \varepsilon+\log2+P(f_0,\ldots,f_{m-1},\varphi).
\end{align*}
Now  $h_{\mu}(f_0,\ldots,f_{m-1},\xi) \leq h_{\mu}(f_0,\ldots,f_{m-1},\eta)+H_{\mu}(\xi | \eta)$ (Theorem \ref{thm:entropy1}$(\rmnum{4})$) so that \[ h_{\mu}(f_0,\ldots,f_{m-1},\xi)+\int \varphi d\mu \leq 2a+\log2+P(f_0,\ldots,f_{m-1},\varphi)\] and hence \[h_{\mu}(f_0,\ldots,f_{m-1})+\int\varphi d\mu \leq 2a+\log2+P(f_0,\ldots,f_{m-1},\varphi).\]
Since $a$ is chosen arbitrarily, we get the desired inequality \[h_{\mu}(f_0,\ldots,f_{m-1})+\int \varphi d\mu \leq \log2+P(f_0,\ldots,f_{m-1},\varphi)\] immediately.
\end{proof}

\textbf{Proof of the Theorem \ref{thm:relation1}.}

According to Lemma \ref{lem:relation2} we have \[h_{\mu}(f_0,\ldots,f_{m-1})+\int \varphi d\mu \leq \log2+P(f_0,\ldots,f_{m-1},\varphi),~~ \forall \mu \in \mathscr{M}(X,f_0,\ldots,f_{m-1}).\] Let $G=\{f_0,\ldots,f_{m-1}\}$, then for any $n \in \mathbb{N}$, we have
\begin{align*}
&\quad h_{\tiny{\overbrace{(\mu\times \mu)\cdots(\mu\times \mu)}^{2^n}}}(\overbrace{(G\times G)\cdots (G\times G)}^{2^n})+\int \overbrace{(\varphi\times\varphi) \cdots (\varphi\times \varphi)}^{2^n}d\overbrace{(\mu\times \mu)\cdots(\mu\times \mu)}^{2^n} \\ & \leq \log2+P(\overbrace{(G\times G)\cdots (G\times G)}^{2^n},\overbrace{(\varphi\times \varphi)\cdots(\varphi\times \varphi)}^{2^n}).
\end{align*}
According to Theorem \ref{prop:properties of pressure2}$(\rmnum{2})$ and Theorem \ref{prop:entropy3}$(\rmnum{2})$ we have
\[2^{n+1}h_{\mu}(G)+2^{n+1}\int \varphi d\mu \leq \log2+2^{n+1}P(G,\varphi).\]
Then \[h_{\mu}(G)+\int \varphi d\mu \leq \frac{1}{2^{n+1}}\log2+P(G,\varphi).\] Let $n \rightarrow \infty$, we get the desired inequality \[h_{\mu}(G)+\int \varphi d\mu \leq P(G,\varphi)\] immediately. \qed

\textbf{Open problem.} Can we get the variational principle for a free semigroup action?

\section{Entropy of a free semigroup action generated by affine transformations}\label{sec:affine}

Now we apply Theorem \ref{thm:relation1} to study the relationship between the Haar measure entropy and the topological entropy of a free semigroup action generated by affine transformations.

\begin{theorem}\label{thm:metr}
Let $X$ be a compact metrizable group, $A_0, \cdots, A_{m-1}$ surjective endomorphisms of $X$ and $a_0, \cdots, a_{m-1} \in X$. Let $\mu$ denote(normalised)Haar measure on $X$ and $d$ a left-invariant metric on $X$. Let $g_i = a_i \cdot A_i$ for any $0 \leq i \leq m-1$. Then
\begin{align*}
&\quad  \lim_{\varepsilon\to 0}\limsup_{n\rightarrow \infty}\left[\frac{1}{n} \left(\frac{1}{m^n} \sum_{|w|=n}\log \frac{1}{\mu(D_w(e, \varepsilon, A_0, \cdots, A_{m-1}))}\right)\right] \\ & \leq h_{\mu}(g_0, \cdots, g_{m-1}) \\& \leq \lim_{\varepsilon\to 0} \limsup_{n\to \infty} \left[\frac{1}{n} \log \left(\frac{1}{m^n} \sum_{|w|=n} \frac{1}{\mu(D_w(e, \varepsilon, A_0, \cdots, A_{m-1}))}\right)\right],
\end{align*}
where
\[D_w(e,\varepsilon , A_0, \cdots, A_{m-1}) = \bigcap_{w' \leq w} A_{w'}^{-1} (B_d(e, \epsilon)),\]
$e$ is the identity element of $X$ and $B_d(e, \varepsilon)$ is the open ball with center $e$ and radius $\varepsilon$ with respect to the metric $d$. Particularly, we have
\begin{align*}
&\quad  \lim_{\varepsilon\to 0}\limsup_{n\rightarrow \infty}\left[\frac{1}{n} \left(\frac{1}{m^n} \sum_{|w|=n}\log \frac{1}{\mu(D_w(e, \varepsilon, A_0, \cdots, A_{m-1}))}\right)\right] \\ & \leq h_{\mu}(A_0, \cdots, A_{m-1}) \\& \leq \lim_{\varepsilon\to 0} \limsup_{n\to \infty} \left[\frac{1}{n} \log \left(\frac{1}{m^n} \sum_{|w|=n} \frac{1}{\mu(D_w(e, \varepsilon, A_0, \cdots, A_{m-1}))}\right)\right].
\end{align*}
\end{theorem}

\begin{proof}
Obviously, $\mu \in \mathscr{M}(X,g_0,\cdots,g_{m-1})$. For any $\varepsilon >0$, $w = i_1 \cdots i_k \in F_m^+$ and $w' = i_l \cdots i_k$ where $1 \leq l \leq k$, from the equation (4.3) in the proof of Theorem 4.2 by Wang and Ma\cite{WMA}, we have
\[g_{w'}^{-1} B_d(g_{w'}(x), \varepsilon) = x \cdot A_{w'}^{-1} B_d(e, \varepsilon).\]
Then
\begin{align*}
D_w(x, \varepsilon, g_0, \cdots, g_{m-1}) & := \bigcap_{w' \leq w} g_{w'}^{-1} (B_d(g_{w'}(x), \varepsilon))\\
&=\bigcap_{w' \leq w} x \cdot A_{w'}^{-1} B_d(e, \varepsilon)\\
&= x \cdot  \bigcap_{w' \leq w} A_{w'}^{-1} B_d(e, \varepsilon)\\
&= x \cdot D_{w}(e, \varepsilon, A_0, \cdots , A_{m-1}).
\end{align*}
and \[\mu(D_w(x, \varepsilon, g_0, \cdots, g_{m-1}))=\mu(D_{w}(e, \varepsilon, A_0, \cdots , A_{m-1})).\]

Let $\varepsilon>0$ and $\xi=\{C_1,\cdots,C_k\}$ be a partition of $X$ into Borel sets of diameter less than $\varepsilon$. For $w \in F_m^{+}$, if $x \in \bigcap_{w'\leq w}g_{w'}^{-1}(C_{i_{w'}})$, then
\[\bigcap_{w'\leq w}g_{w'}^{-1}(C_{i_{w'}}) \subset x\cdot D_{w}(e, \varepsilon, A_0, \cdots , A_{m-1})\]
where $C_{i_{w'}} \in \xi$.  Since if $y \in \cap_{w'\leq w}g_{w'}^{-1}(C_{i_{w'}})$, then $g_{w'}(x),g_{w'}(y) \in C_{i_{w'}}$ and hence $y \in g_{w'}^{-1}B(g_{w'}(x),\varepsilon), \forall w'\leq w$, then
\begin{align*}
y \in \bigcap_{w'\leq w}g_{w'}^{-1}B(g_{w'}(x),\varepsilon)=D_w(x, \varepsilon, g_0, \cdots, g_{m-1})=x\cdot D_w(e, \varepsilon, g_0, \cdots, g_{m-1}).
\end{align*}
Thus $\mu(\bigcap_{w'\leq w}g_{w'}^{-1}(C_{i_{w'}})) \leq  \mu(D_{w}(e, \varepsilon, A_0, \cdots , A_{m-1}))$ and taking logarithms we see that
\begin{equation}\label{equ:meas}
\begin{aligned}
 &\quad \sum_{w'\leq w, i_{w'} \in \{1,\cdots,k\}}\mu(\bigcap_{w'\leq w}g_{w'}^{-1}(C_{i_{w'}}))\log\mu(\bigcap_{w'\leq w}g_{w'}^{-1}(C_{i_{w'}})) \\ & \leq \sum_{w'\leq w, i_{w'} \in \{1,\cdots,k\}}\mu(\bigcap_{w'\leq w}g_{w'}^{-1}(C_{i_{w'}}))\log\mu(D_{w}(e, \varepsilon, A_0, \cdots , A_{m-1})) \\
&=\log\mu(D_{w}(e, \varepsilon, A_0, \cdots , A_{m-1})) .
\end{aligned}
\end{equation}

Therefore
\begin{align*}
&\quad h_{\mu}(g_0, \cdots, g_{m-1}) \\& \geq h_{\mu}(g_0, \cdots, g_{m-1},\xi) \\
&=\lim_{n\rightarrow \infty}\frac{1}{n}\left[\frac{1}{m^n}\sum_{|w|=n}H_{\mu}(\bigvee_{w' \leq w}g_{w'}^{-1}\xi)\right] \\ & \geq \limsup_{n\rightarrow \infty}\left[\frac{1}{n} \left(\frac{1}{m^n} \sum_{|w|=n}\log \frac{1}{\mu(D_w(e, \varepsilon, A_0, \cdots, A_{m-1}))}\right)\right]~~ by~(\ref{equ:meas}).
\end{align*}
 Since $\varepsilon$ is arbitrary, we have
\[h_{\mu}(g_0, \cdots, g_{m-1}) \geq \lim_{\varepsilon\to 0}\limsup_{n\rightarrow \infty}\left[\frac{1}{n} \left(\frac{1}{m^n} \sum_{|w|=n}\log \frac{1}{\mu(D_w(e, \varepsilon, A_0, \cdots, A_{m-1}))}\right)\right].\]
By Wang, Ma and Lin\cite{WMAL}, we have
\begin{align*}
&\quad h(g_0, \cdots, g_{m-1})=h(A_0, \cdots, A_{m-1})\\&=\lim_{\epsilon\to 0} \limsup_{n\to \infty} \left[\frac{1}{n} \log \left(\frac{1}{m^n} \sum_{|w|=n} \frac{1}{\mu(D_w(e, \varepsilon, A_0, \cdots, A_{m-1}))}\right)\right],
\end{align*}
and by Theorem \ref{thm:relation1} we have \[h_{\mu}(g_0, \cdots, g_{m-1}) \leq h(g_0, \cdots, g_{m-1}).\]
Hence
\begin{align*}
&\quad  \lim_{\varepsilon\to 0}\limsup_{n\rightarrow \infty}\left[\frac{1}{n} \left(\frac{1}{m^n} \sum_{|w|=n}\log \frac{1}{\mu(D_w(e, \varepsilon, A_0, \cdots, A_{m-1}))}\right)\right] \\ & \leq h_{\mu}(g_0, \cdots, g_{m-1}) \\& \leq \lim_{\varepsilon\to 0} \limsup_{n\to \infty} \left[\frac{1}{n} \log \left(\frac{1}{m^n} \sum_{|w|=n} \frac{1}{\mu(D_w(e, \varepsilon, A_0, \cdots, A_{m-1}))}\right)\right].
\end{align*}
\end{proof}

\begin{remark}
In Theorem \ref{thm:metr}, if $m=1$, let $g=a\cdot A$, we have \[h_{\mu}(g)=h_{\mu}(A)=\lim_{\varepsilon\to 0}\limsup_{n\rightarrow \infty}\left[\frac{1}{n} \log\frac{1}{\mu(\bigcap_{i=0}^nA^{-i} B_d(e, \varepsilon))}\right]=h(A)=h(g),\] which had been proved by Bowen\cite{Bowen,WALTER2}.
\end{remark}

{\bf Acknowledgement.} The authors really appreciate the referees’ valuable remarks and suggestions that helped a lot. Dongkui Ma was supported by Guangdong Natural Science Foundation 2014A030313230 and "Fundamental Research Funds for the Central Universities" SCUT(2015ZZ055, 2015ZZ127). Dongkui Ma is the corresponding author.

\bibliographystyle{amsplain}

\begin{thebibliography}{10}

\bibitem{AKM} R. Adler, A. G. Konheim and J. McAndrew, {\it Topological entropy}, Trans. Amer. Math. Soc. 114(1965), 309-319.
\bibitem{BAR} L. Barreira, {\it A non-additive thermodynamic formulism and application to dimension theory of hyperbolic dynamical systems}, Ergodic Theory Dynam. Systems, 16(1996), 871-928.
\bibitem{BIS1} A. Bi\'{s}, {\it Entropies of a semigroup of maps}, Discrete Contin. Dyn. Systs. Series A, 11(2004), 639-648.
\bibitem{BIS3} A. Bi\'{s}, {\it Partial variational principle for finitely generated groups of polynomial growth and some foliated spaces}, Colloq. Math. 110 (2008), no. 2, 431–449.
\bibitem{BIS4} A. Bi\'{s}, {\it An analogue of the variational principle for group and pseudogroup actions}, Ann. Inst. Fourier (Grenoble) 63 (2013), no. 3, 839–863.
\bibitem{BIS2} A. Bi\'{s} and  M. Urba\'{n}ski, {\it Some remarks on topological entropy of a semigroup of continuous maps}, Cubo, 8(2006), 63-71.
\bibitem{Bowen} R. Bowen, {\it Entropy for group endomorphisms and homogeneous spaces}, Trans. Amer. Math. Soc. 153(1971), 401-414; erratum: Trans. Amer. Math. Soc. 181(1971), 509-510.
\bibitem{BUfE} A. Bufetov, {\it Topological entropy of free semigroup actions and skew-product transformations}, J. Dynam. Control Systems 5 (1999), no. 1, 137-143.
\bibitem{BUGR} B. Burago, {\it Semi-dispersing billiards of infinite topological entropy}, Ergodic Theory Dynam. Systems, 26 (2006), 45-52.
\bibitem{CANOVAS} J. S. C\'{a}novas, {\it On entropy of non-autonomous discrete systems}, Progress and challenges in dynamical systems, Springer Proc. Math. Stat., 54, Springer Heidelberg, (2013) , 143-159.
\bibitem{CARVALHO} M. Carvalho, F. B. Rodrigues, P. Varandas, {\it Semigroup actions of expanding maps}, preprint, (2016), arXiv:1601.04275v1.
\bibitem{CAO} Y. Cao, H. Hu, Y. Zhao, {\it Nonadditive measure-theoretic pressure and applications to dimensions of an ergodic measure}, Ergodic Theory Dynam. Systems, 33(2013), no.3, 831-850.
\bibitem{CHUNG} N. Chung, {\it Topological pressure and the variational principle for actions of sofic groups}, Ergodic Theory Dynam. Systems, 33(2013),1363--1390.
\bibitem{CONZE} J.P. Conze,  {\it Entropie d'un groupe abelien de transformations}, Z.Wahrscheinlichkeitstheorie und verwandte Gebiete. 25(1972) N0.1, 11-30.
\bibitem{Dinaburg} E. I. Dinaburg, {\it The relation between topological entropy and metric entropy}, Soviet Math. Dokl. 11(1970), 13-16.
\bibitem{FRI} S. Friedland, {\it Entropy of graphs, semigroups and groups}, London Math. Soc. Lecture Note Ser., 228, Cambridge Univ. Press, Cambridge, 1996. 319-343.
\bibitem{GLW} E. Ghys, R. Langevin and P. Walczak, {\it Entropie geometrique des feuilletages}, Acta Math. 160(1988), 105-142.
\bibitem{HU} H. Hu, {\it Some ergodic properties of commuting diffeomorphisms}, Ergodic Theory Dynam. Systems, 13(1993), 73-100.
\bibitem{HUANG} W. Huang, Y. Yi, {\it A local variational principle of pressure and its applications to equilbrium states}, Israel J. Math. 161(2007), 29-74.
\bibitem{KAT} Y. Katznelson, B. Weiss, {\it Commuting measure-preserving transformaions}, Israel J. Math. 12(1972), No. 2, 161-173.
\bibitem{KIR} A.A. Kirillov, {\it Dynamical systems, factors and group representations}, Math. Surv. 22(1967), No. 5, 67-80.
\bibitem{KOL} S. Kolyada, L. Snoha {\it topological entropy of nonautonomous dynamical systems}, Random Comput Dyn. 4(1996), 205-33.
\bibitem{LSHMA} D. Ma, S. Liu, {\it Some properties of topological pressure of a semigroup of continuous maps}, Dyn. Syst. 29(2014), 1-17.
\bibitem{MAWU} D. Ma, M. Wu, {\it Topological pressure and topological entropy of a semigroup of maps}, Discrete Contin. Dyn. Systs. 31(2011), 545-557.
\bibitem{MISI} M. Misiurewicz, {\it A short proof of the variational principle for a $\mathbb{Z}_{+}^N$ action on a compact space}, Asterique. 40(1976), 147-187.
\bibitem{PESIN} Y. Pesin and B. S. Pitskel', {\it Topological pressure and its variational principle for noncompact sets}, functional Analysis and its Applications, 18(1984), 307-318.
\bibitem{RODRIGUES} F. B. Rodrigues, P. Varandas, {\it Specification and thermodynamical properties of semigroup actions}, preprint, (2015), arXiv:1502.01163v2.
\bibitem{RUE} D. Ruelle, {\it Thermodynamic Formalism}, Addison-Wesley, Reading, MA 1978.
\bibitem{TANG} J. Tang, B. Li, W.C. Cheng, {\it Some properties on topological entropy of free semigroup action}, preprint, 2015.
\bibitem{THOM} D. Thompson, {\it A thermodynamic definition of topological pressure for non-compact sets}, Ergodic Theory Dynam. Systems, 31 (2011), 527-547.
\bibitem{WMA} Y. Wang, D. Ma, {\it On the topological entropy of a semigroup of continuous maps}, J. Math. Anal. Appl. 427(2015) , 1084-1100.
\bibitem{WMAL} Y. Wang, D. Ma, X. Lin, {\it On the topological entropy of free semigroup actions},  J. Math. Anal. Appl. 435(2016), 1573-1590.
\bibitem{WALTER1} P. Walters, {\it A variational principle for the pressure of continuous transformations}, Amer. J. Math. 97(1975), 937-971.
\bibitem{WALTER2} P. Walters, {\it An introduction to ergodic theory}, Springer-Verlag, New York, Heidelberg, Berlin, 1982.
\bibitem{ZHANG} G. Zhang, {\it Variational principles of pressure}, Discrete Contin. Dyn. Syst. 24(2009), No.4, 1409-1435.


\end{thebibliography}

\end{document}